\documentclass{article}
\pdfoutput=1
\usepackage{lineno}
\usepackage[a4paper]{geometry}
\usepackage{graphicx} 
\usepackage{amssymb}
\usepackage{amsthm}
\usepackage{xcolor}
\usepackage{amsmath}
\usepackage{bm}
\usepackage{algorithm,algpseudocode}
\usepackage[colorlinks = true, urlcolor = black, linkcolor = black, citecolor = black]{hyperref}
\usepackage{cleveref}
\usepackage{caption}
\usepackage{subcaption}
\usepackage{mathtools}
\usepackage{microtype}
\usepackage[normalem]{ulem}
\usepackage{overpic}
\usepackage[utf8]{inputenc}

\algnewcommand\algorithmicinput{\textbf{Input:}}
\algnewcommand\Input{\item[\algorithmicinput]}
\algnewcommand\algorithmicoutput{\textbf{Output:}}
\algnewcommand\Output{\item[\algorithmicoutput]}

\newtheorem{theorem}{Theorem}
\newtheorem{lemma}[theorem]{Lemma}

\newtheorem{remark}{Remark}

\newcommand{\R}{\mathbb R}
\newcommand{\N}{\mathbb N}

\newcommand{\range}{\mathrm{range}}
\newcommand{\HS}{\mathrm{HS}}
\newcommand{\Id}{\mathrm{Id}}
\newcommand{\F}{\mathrm{F}}
\newcommand{\op}{\mathrm{op}}
\DeclareMathOperator{\diag}{diag}

\DeclareMathOperator{\Tr}{Trace}
\DeclareMathOperator{\T}{Tr}
\DeclareMathOperator{\spann}{span}
\newcommand{\rev}[1]{{\color{black} #1}}

\numberwithin{equation}{section}

\title{On the randomized SVD in infinite dimensions}
\author{Daniel Kressner\thanks{Institute of Mathematics, EPF Lausanne, 1015 Lausanne, Switzerland.} \and David Persson\thanks{Courant Institute of Mathematical Sciences, New York University, New York, NY 10012, USA \& Center for Computational Mathematics, Flatiron Institute, New York, NY 10010, USA.} \and Andr\'e Uschmajew\thanks{Institute of Mathematics \& Centre for Advanced Analytics and Predictive Sciences, University of Augsburg, 86159 Augsburg, Germany.}
}
\date{}

\begin{document}

\maketitle

\begin{abstract}
Randomized methods, such as the randomized SVD (singular value decomposition) and Nystr\"om approximation, are an effective way to compute low-rank approximations of large matrices. Motivated by applications to operator learning, Boull\'e and Townsend (FoCM, 2023) recently proposed an infinite-dimensional extension of the randomized SVD for a Hilbert--Schmidt operator $\mathcal A$ that invokes randomness through a Gaussian process with a covariance operator $\mathcal K$. While the non-isotropy introduced by $\mathcal K$ allows one to incorporate prior information on $\mathcal A$, an unfortunate choice may lead to unfavorable performance and large constants in the error bounds. In this work, we introduce a novel infinite-dimensional extension of the randomized SVD that does not require such a choice and enjoys error bounds that match those for the finite-dimensional case. 
\rev{Our extension implicitly uses isotropic random vectors, reflecting a choice commonly made in the finite-dimensional case.}
In fact, the theoretical results of this work show how the usual randomized SVD applied to a discretization of $\mathcal A$ approaches our infinite-dimensional extension as the discretization gets refined, both in terms of error bounds and the Wasserstein distance.
We also present and analyze a novel extension of the Nystr\"om approximation for self-adjoint positive semi-definite trace class operators.
\textbf{MSC Classes.} 65F55, 65F60, 68W20
\end{abstract}

\section{Introduction}

Randomized methods, such as the randomized SVD~\cite{rsvd} and (generalized) Nystr\"om methods~\cite{gittensmahoney,nakatsukasa2020fast,Tropp2017}, are an effective way to compute a low-rank approximation of a (large) matrix $\rev{A} \in \R^{m\times n}$. Given a random matrix $\Omega \in \R^{n \times r}$ with $r \ll n$, the basic form of the randomized SVD proceeds by first computing the sketch $Y = \rev{A} \Omega$. Using a QR factorization of $Y$, it then computes an orthonormal basis 
$Q \in \R^{m \times r}$ and returns the rank-$r$ approximation
\begin{equation}
 \label{eq:randsvd}
 \rev{A} \approx Q \big( \rev{A}^* Q \big)^*
\end{equation}
in which the columns of $\rev{A}$ have been projected onto the range of $Y$. Let $r = k + p$ for a small oversampling parameter $p \ge 4$. Then the analysis in~\cite[Sec. 10]{rsvd} shows that
the approximation error of~\eqref{eq:randsvd} is not significantly larger than the \emph{best} rank-$k$ approximation error, with high probability,
when choosing $\Omega$ to be a Gaussian random matrix. As discussed in~\cite{randlapack,rsvd}, other random matrix models can lead to faster algorithms,  but they typically result in weaker theoretical guarantees.
 
This paper is concerned with an infinite-dimensional extension of the randomized SVD and Nyström approximation. 
Concretely, we consider a linear operator $\mathcal A: H_1 \to H_2$ between separable Hilbert spaces $H_1,H_2$. The prototypical example is an integral operator of the form
\begin{equation} \label{eq:integraloperator}
 \mathcal A: L^2(D_y) \to L^2(D_x) \quad \mathcal A: f(y) \mapsto g(x):= \int_{D_y} \kappa(x,y) f(y)\,\mathrm{d} y
\end{equation}
on some domains $D_x, D_y$.
If the integral kernel satisfies $\kappa \in L^2(D_x \times D_y)$ then $\mathcal A$ is a Hilbert--Schmidt operator and admits an infinite SVD~\cite{functionaldatanalysis}, which in fact corresponds to an SVD of the integral kernel $\kappa(x,y)$~\cite{Schmidt1907}. Low-rank approximations of $\mathcal A$ play an important role in the context of Gaussian process regression \cite{bach2017equivalence}, PDE learning \cite{boulle2022learning}, and Bayesian inverse problems \cite{saibabarandom,alexanderian2014optimal}.
In practice, the randomized SVD is of course applied to a discretization of $\mathcal A$. Nevertheless, from a theoretical point of view, it is of interest 
to study an appropriate infinite-dimensional extension of the randomized SVD that applies directly to $\mathcal A$. Among others, this can provide important insights into the behavior of the randomized SVD as the discretization gets refined.

It seems plausible that any sensible extension of the randomized SVD starts with sketching $\mathcal A$ by applying it to $r$ independent random elements from the Hilbert space  $H_1$. This raises the question of what is an appropriate extension of a Gaussian random vector in $H_1$.
When $H_1$ is a function space, like in~\eqref{eq:integraloperator}, it is natural to draw the random elements from a Gaussian process with zero mean and covariance operator $\mathcal K: H_1 \to H_1$. This is the approach proposed and analyzed by Boull\'e and Townsend in~\cite{boulle2022learning}; see also \cite{perssonboullekressner2025}.
Both, from a theoretical and practical point of view, a certain degree of regularity / eigenvalue decay must be imposed on $\mathcal K$. Therefore, the discretization of such a Gaussian process inevitably yields non-standard multivariate Gaussian random vectors. The use of such non-isotropic random vectors affects the analysis of the randomized SVD in finite and infinite dimensions. In particular, the stochastic error bounds~\cite[Theorem 1]{boulle2022learning},~\cite[Theorem 2]{boulleiclr}, and~\cite[Theorem 3.3]{perssonboullekressner2025} feature additional constants that quantify the interaction between $\mathcal K$ and the right singular vectors\footnote{\rev{Throughout this work, we use the term \emph{vector} to refer to an element in a Hilbert space, even when the underlying space is a function space, such as $L^2(D)$ for some domain $D$.}} of $\mathcal A$. On the one hand, one might argue, as in~\cite[Sec. 3]{boulleiclr}, that the presence of $\mathcal K$ allows one to incorporate prior information on $\mathcal A$ and the experiments in~\cite[Sec. 5.1]{boulleiclr} demonstrate some benefit of doing so. On the other hand, non-standard multivariate Gaussian random vectors can be more expensive to sample, sometimes significantly so. More importantly, imposing too much regularity on $\mathcal K$ (relative to $\mathcal A$) leads to unfavorable constants in the error bounds and degrades the performance of the randomized SVD. Indeed, the experiments in~\cite[Sec. 5.2]{boulleiclr} clearly reveal the importance of choosing a covariance operator $\mathcal K$ with slow eigenvalue decay. 

It is important to stress that the randomized SVD is routinely applied with \emph{isotropic}
Gaussian random vectors in the context of discretized infinite-dimensional linear operators; see~\cite{Levitt2024,Lin2011,Lin2017,Saibaba2021} for examples.
Even Chebfun \cite{chebfun} deviates from its principles and does not impose decay on the Chebyshev coefficients when sketching with random Chebfun objects.\footnote{See~\url{https://www.chebfun.org/examples/ode-eig/ContourProjEig.html} for an example.}
The main goal of the present work is to derive a framework that captures these uses of the randomized SVD in the infinite-dimensional limit.

\rev{\subsection{Contribution and outline}}
\rev{The contributions of this work can be summarized as follows:
\begin{enumerate}
    \item[(i)] In \Cref{section:randomizedSVD}, we show that the standard randomized SVD for matrices, which uses isotropic random vectors, admits a natural infinite-dimensional extension to Hilbert--Schmidt operators $\mathcal A: H_1 \to H_2$. This version, presented as \Cref{alg:infrsvd}, implicitly works with isotropic Gaussian random vectors by sampling the columns of the \emph{sketch} (instead of $\Omega$) from the centered Gaussian measure in $H_2$ associated with the covariance operator $\mathcal A \mathcal A^*$. This bypasses the need for choosing a covariance kernel $\mathcal K$ in $H_1$ discussed above and perfectly matches the (standard) randomized SVD in finite dimensions. Indeed, our error analysis (\Cref{lemma:structural} and \Cref{thm:InfRandSVD}) reproduces the results from~\cite[Sec. 10]{rsvd} without any additional constants, in contrast to earlier work~\cite{boulle2022learning,perssonboullekressner2025}.
    \item[(ii)] In \Cref{section:fdapprox}, we study the interplay between discretization and the randomized SVD for the simple case when the operator $\mathcal A$ is compressed to finite dimensions based using given orthonormal bases. We derive a bound that decouples the discretization error from the randomized SVD error (\Cref{thm:main}). These bounds illustrate that---as the discretization is refined---one recovers the classical results from~\cite[Sec. 10]{rsvd} and they motivate the design of a practical, adaptive algorithm, which we present in \Cref{sec:adaptive}.
\item[(iii)] Specifically, in \Cref{thm:convergenceexpect} we prove that the randomized SVD applied to the discretized operator approaches the proposed infinite-dimensional extension from (i) in the Wasserstein distance as the discretization gets refined. 
\item[(iv)] In \Cref{section:nystrom}, we discuss an infinite-dimensional generalization of the Nyström approximation, suitable for self-adjoint positive semi-definite  trace class operators. 
\item[(v)] Finally,~\Cref{sec:numexp} presents numerical experiments illustrating our developments.
\end{enumerate}

Let us reiterate that the main contribution of this work is on the conceptual level; it establishes a connection between the randomized SVD in finite and infinite dimensions, and provides a theoretical justification for what is already common practice.}

\section{Preliminaries}

In this section, we recall well-known  mathematical preliminaries on Hilbert--Schmidt operators, quasi-matrices, and infinite-dimensional analogues of Gaussian random vectors.

\subsection{Hilbert--Schmidt operators and quasi-matrices}\label{section:HS}

Throughout this work, we consider a linear operator $\mathcal{A} : H_1 \to H_2$ between two separable Hilbert spaces $(H_1,\langle\cdot,\cdot\rangle_{H_1})$ and $(H_2,\langle\cdot,\cdot\rangle_{H_2})$.
To avoid trivial cases, we assume that $H_1$ and $H_2$ are infinite-dimensional. Moreover, in this work Hilbert spaces are always defined over $\mathbb{R}$. These assumptions imply that $H_1$ and $H_2$ admit countable orthonormal bases. We say that $\mathcal{A}$ is a Hilbert--Schmidt operator when $\mathcal{A}$ has finite Hilbert--Schmidt norm \cite[Definition 4.4.2]{functionaldatanalysis}:
\begin{equation*}
    \|\mathcal{A}\|_{\HS} := \Big(\sum\limits_{i} \|\mathcal{A}e_i \|_{H_2}^2\Big)^{1/2}<\infty,
\end{equation*}
where $\{e_i\}_i$ is \emph{any} orthonormal basis of $H_1$ and $\|\cdot\|_{H_2}$ is the norm induced by $\langle\cdot,\cdot\rangle_{H_2}$.
Hilbert--Schmidt operators are compact \cite[Theorem~4.4.2]{functionaldatanalysis} and therefore admit an SVD \cite[Theorem~4.3.5]{functionaldatanalysis} of the form
\begin{equation}\label{eq:longsvd}
    \mathcal{A} = \sum\limits_{i} \sigma_i\langle v_i,\cdot\rangle_{H_1} u_i.
\end{equation}
The scalars $\sigma_1\geq \sigma_2\geq \cdots\geq 0$ are the singular values of $\mathcal{A}$, the vectors $\{u_i\}_i$ form an orthonormal system of left singular vectors in $H_2$, whereas the vectors $\{v_i\}_i$ from an orthonormal system of right singular vectors in $H_1$. The singular values allow us to express the  Hilbert--Schmidt norm of $\mathcal{A}$ as
\begin{equation}\label{eq:HS2}
    \|\mathcal{A}\|_{\HS} =\Big(\sum_{i} \sigma_i^2\Big)^{1/2},
\end{equation}
which is an infinite-dimensional analogue of the Frobenius norm. In addition, the operator norm of $\mathcal A$ satisfies $$\|\mathcal{A}\|_{\op} := \sup\limits_{ \|f\|_{H_1} = 1}\|\mathcal{A}f\|_{H_2} = \sigma_1.$$

Quasi-matrices are a convenient tool for adopting matrix-like notations when dealing with collections of elements in a Hilbert space;~see~\cite{quasimatrix}. In particular, they allow us to express the SVD~\eqref{eq:longsvd} in a more common and compact form.
For this purpose, we introduce the operator
\[
 U:\ell^2 \to H_2, \quad U: \xi \mapsto \sum\limits_{i} \xi_i u_i,
\]
where $\ell^2$ is the space of square-summable sequences $\xi = \{ \xi_i\}_i$. We can interpret $U$ as an infinite \emph{quasi-matrix} containing the left singular vectors of $\mathcal{A}$ as columns, which we write as $U = \begin{bmatrix} u_1, & u_2, & \cdots \end{bmatrix}$. Defining $V:\ell^2 \to H_2$ analogously and introducing the diagonal operator $\Sigma = \diag(\sigma_1,\sigma_2,\ldots): \ell^2 \to \ell^2$, which multiplies the $i$th entry of a sequence with $\sigma_i$, we can 
write~\eqref{eq:longsvd} as the composition of these three operators:
\[
 \mathcal{A} = U \Sigma V^*,
\]
where 
$V^* : H_1 \to \ell^2$ denotes the adjoint operator, defined as $V^* f = \{\langle v_i, f \rangle_{H_1}\}_i \in \ell^2$.

We will also use quasi-matrices with a finite number of columns. In particular, the quasi-matrix $U_k = \begin{bmatrix} u_1, & \cdots &, u_k \end{bmatrix}$ containing the first $k$ left singular vectors as columns represents the operator $U_k:\R^k\to H_2$, $\xi \mapsto \xi_1 u_1 + \dots + \xi_k u_k$ where $\xi_i$ is the $i$-th entry of $\xi \in \mathbb R^k$. Let $U_k^\perp:\ell^2\to H_2$ denote the quasi-matrix $U_k^\perp = \begin{bmatrix} u_{k+1}, & u_{k+2}, & \cdots \end{bmatrix}$ containing the remaining left singular vectors, and define $V_k$ and $V_k^\perp$ similarly for the right singular vectors. Introducing the diagonal matrix $\Sigma_k = \diag(\sigma_1,\ldots,\sigma_k)$ and operator $\Sigma_k^\perp = \diag(\sigma_{k+1},\sigma_{k+2},\ldots)$ allows us express the SVD~\eqref{eq:longsvd} as
\begin{equation*}
    \mathcal{A} = U_k \Sigma_k V_k^* + U_k^\perp \Sigma_k^\perp (V_k^\perp)^*,
\end{equation*}
where $V_k^*: H_1 \to \mathbb R^k$, $(V_k^\perp)^* : H_1 \to \ell^2$ are the adjoints of $V_k, V_k^\perp$ defined as above.
In particular, $V_k^* f = [\langle v_1, f \rangle_{H_1},\cdots,\langle v_k, f \rangle_{H_1}]^* \in \mathbb R^k$.

Finally, for a quasi-matrix $F = \begin{bmatrix} f_1, & f_2, & \cdots \end{bmatrix}$, either finite or infinite, we denote by $\mathcal{A}F = \begin{bmatrix} \mathcal A f_1, & \mathcal A f_2, & \cdots \end{bmatrix}$ the quasi-matrix whose $i$-th column is $\mathcal{A}f_i$.

By the the Schmidt-Mirsky theorem, the operator $U_k \Sigma_k V_k^*$ is an optimal approximation to $\mathcal{A}$ in the Hilbert--Schmidt norm by an operator of rank at most $k$, with the error
\[
\| \mathcal A - U_k \Sigma_k V_k^* \|_{\HS}^2 = \sum_{j > k} \sigma_j^2;
\]
see,~e.g.,~\cite[Theorem~4.4.7]{functionaldatanalysis}.

\subsection{Trace-class operators}

Let $\mathcal T : H \to H$ be a Hilbert--Schmidt operator on a separable Hilbert space $(H,\langle\cdot,\cdot\rangle_H)$ with SVD $\mathcal T = U \Sigma V^*$. Then $\mathcal T$ is called a trace-class or (nuclear) operator if it has finite trace norm $\|\mathcal{T}\|_{\T}:=\sigma_1 + \sigma_2 + \cdots < \infty$. In this case, the trace of $\mathcal T$ is well defined and given by
\begin{equation}\label{eq:tracedef}
    \Tr(\mathcal{T}):= \sum\limits_{i} \langle e_i, \mathcal{T}e_i \rangle_{H},
\end{equation}
where $\{e_i\}_i$ is an orthonormal basis of $H$ and the value of the trace is independent of the choice of this basis. 
By the Schmidt-Mirsky theorem, the rank-$k$ approximation $U_k \Sigma_k V_k^*$ to $\mathcal T$ is also optimal in the trace norm, with the error
\[
\| \mathcal T - U_k \Sigma_k V_k^* \|_{\T} = \sigma_{k+1} + \sigma_{k+2} + \cdots.
\]

We will focus on the special case when \rev{$\mathcal{T}$} is self-adjoint positive semi-definite (SPSD). In this case, one can choose $U = V$, the singular values coincide with the eigenvalues, and 
\begin{equation}\label{eq:trace}
    \|\mathcal{T}\|_{\T} = \Tr(\mathcal{T}).
\end{equation}
For an SPSD trace-class operator $\mathcal{T}$ one can define the (principal) square root $\mathcal{T}^{1/2}$ through the spectral decomposition $\mathcal{T}^{1/2} = U \Sigma^{1/2} U^*$, where $\Sigma^{1/2} = \diag(\sigma_1^{1/2},\sigma_2^{1/2},\ldots)$. 

Finally, we note that for any Hilbert--Schmidt operator $\mathcal A : H_1 \to H_2$, the operators $\mathcal A^* \mathcal A$ and $\mathcal A \mathcal A^*$ are SPSD trace-class operators on $H_1$ and $H_2$, respectively, and satisfy, in view of \eqref{eq:trace} and \eqref{eq:HS2},
\begin{equation*}
    \|\mathcal{A}\|_{\HS} = \Tr(\mathcal{A}^*\mathcal{A})^{1/2} = \Tr(\mathcal{A}\mathcal{A}^*)^{1/2}.
\end{equation*} 

\subsection{Random elements in Hilbert spaces}\label{section:random}

In the sequel, we will consider Gaussian random variables in a separable Hilbert space $H$. Recall that for every SPSD trace class operator $\mathcal T : H \to H$ there exists a unique centered Gaussian measure $\mu_{\mathcal T}$ on $H$ defined on the Borel sets $\mathcal B(H)$ with mean zero and covariance operator $\mathcal T$~\cite[Theorem~1.12]{DaPrato2006}. We will say that a random variable $z$ with values in $H$ whose distribution is the measure $\mu_{\mathcal T}$ is distributed as $\mathcal{N}_{H}(0,\mathcal{T})$. The main property of this measure that we will use is the following~\cite[Corollary~1.19 \& Example~1.22]{DaPrato2006}. Let $u_1,u_2,\dots$ be an orthonormal basis of eigenvectors of $\mathcal T$ with corresponding eigenvalues $\lambda_1 \ge \lambda_2 \ge \dots \ge 0$. If we write the basis expansion of an $\mathcal{N}_{H}(0,\mathcal{T})$ distributed random variable $z$ in the form
\begin{equation}\label{eq: series}
z = \sum\limits_{i} \omega_i \lambda_i^{1/2} u_i
\end{equation}
then for every $n$ the coefficients $\omega_1,\dots,\omega_n$ are i.i.d. $\mathcal N(0,1)$ variables.

We can view the whole sequence $\omega = (\omega_1,\omega_2,\dots) \in \mathbb R^{\mathbb N}$ in~\eqref{eq: series} as an infinite-dimensional analog of a standard Gaussian vector. This notion can be given a precise meaning by equipping $\mathbb R^{\mathbb N}$ with the countable product $\mathbb P$ of centered standard Gaussian measures on $\R$.  This turns $(\mathbb R^{\mathbb N},\mathcal B(\mathbb R^{\mathbb N}), \mathbb P)$ into a probability space, where $\mathcal B(\mathbb R^\N)$ is the Borel $\sigma$-algebra generated by cylindrical sets; see, e.g.,~\cite[2.3.5. Example]{gaussian} or~\cite[Theorem~1.9]{DaPrato2006}. Similar to finite Gaussian vectors, such an infinite sequence $\omega = (\omega_1,\omega_2,\dots)$ of i.i.d.~Gaussians has the property, that any projection onto finitely many components $\omega_i$ yields a random vector with standard normal distribution. Furthermore it follows from Kolmogorov's two series theorem~\cite[1.1.4 Theorem]{gaussian} that the series~\eqref{eq: series} converges both almost surely in $H$ with respect to the measure $\mathbb P$ on $\mathbb R^{\mathbb N}$, as well as in $L_2(\mathbb R^{\mathbb N}, \mathbb P; H)$. We could hence interpret~\eqref{eq: series} as a random variable defined on $(\mathbb R^{\mathbb N},\mathcal B(\mathbb R^{\mathbb N}), \mathbb P)$ with values in~$H$. Since the partial sums $\sum_{i=1}^n \omega_i \lambda_i^{1/2} u_i$ are centered Gaussian variables in $H$ with covariance operator $P_n \mathcal T P_n$, where $P_n$ is the orthogonal projection onto $\spann\{u_1,\dots,u_n\}$, it follows by~\cite[Proposition~1.16]{DaPrato2006} that the random variable $z$ defined in this way has the distribution $\mathcal N(0,\mathcal T)$. We refer to~\cite{gaussian} and~\cite{DaPrato2006} for more details on Gaussian measures and random variables in Hilbert space.

\begin{remark}
The series representation~\eqref{eq: series} is also called the Karhunen-Lo\`eve expansion of a random vector drawn from a Gaussian measure on the Hilbert space $H$~\cite[Theorem 6.19]{stuart2010inverse} and relates to many important special cases. For instance, when $H = L^2(D)$ for some bounded domain $D \subset \mathbb{R}^d$, then~\eqref{eq: series} defines a Gaussian process. The covariance operator $\mathcal T$ of that Gaussian process is an integral operator with the kernel defined defined by the Mercer expansion
\begin{equation*}
    \kappa(x,y) = \sum\limits_{i} \lambda_i u_i(x) u_i(y), \quad x,y \in D;
\end{equation*}
see,~e.g.~\cite[Sec.~3]{adler}.
\end{remark}

\section{Randomized SVD in Hilbert spaces}\label{section:randomizedSVD}

Our extension of the randomized SVD to infinite dimensions is based on the following simple observation.
Given a matrix $\rev{A} \in \R^{m\times n}$ and a Gaussian random matrix $\Omega \in \R^{n\times r}$, the columns of the sketch $Y = A\Omega$ are independent non-standard multivariate Gaussian random vector with mean zero and 
covariance matrix $A A^*$. It turns out that this viewpoint, that the columns of $Y$ are drawn from 
$\mathcal{N}(0,AA^*)$, can be conveniently generalized to the Hilbert space setting
using the construction of random vectors from \Cref{section:random}.

Specifically, consider a Hilbert--Schmidt operator $\mathcal A : H_1 \to H_2$ with singular values $\{\sigma_i\}_i$ and left singular vectors $\{u_i\}_i$. Noting that $\{\sigma_i\}_i\in \ell^2$, the discussion from \Cref{section:random} allows us to define the unique centered Gaussian measure on $H_2$ with covariance operator $\mathcal A \mathcal A^*$. Random vectors $y \in H_2$ distributed according to this Gaussian measure can be written as
\begin{equation}\label{eq:kl_expansion}
 y = \sum_{i} \omega_i \sigma_i u_i, \qquad \omega_i = \frac{1}{\sigma_i} \langle y,u_i \rangle_{H_2} \sim \mathcal N(0,1) \text{ i.i.d.}
\end{equation}
Following our previous terminology, we will say that any vector with the same distribution as $y$ is distributed as $\mathcal{N}_{H_2}(0,\mathcal{A} \mathcal{A}^*)$.

\subsection{Algorithm}

Based on the idea presented above, we arrive at~\Cref{alg:infrsvd}, a direct extension of the randomized SVD to a Hilbert--Schmidt operator. 
\begin{algorithm}[t]
\small
\caption{Infinite-dimensional randomized SVD}
\label{alg:infrsvd}
\begin{algorithmic}[1]
\Input
Hilbert--Schmidt operator $\mathcal{A}:H_1 \to H_2$, target rank $k$, oversampling parameter $p$.
\Output
Rank $k+p$ approximation $\widehat{\mathcal{A}}:= Q Q^* \mathcal{A}$ to $\mathcal{A}$ in factored form.
\State Sample a random quasi-matrix $Y = \begin{bmatrix} y_1, & \cdots &, y_{k+p} \end{bmatrix}$ with $y_j \sim \mathcal{N}_{H_2}(0,\mathcal{A} \mathcal{A}^*)$ i.i.d., as in~\eqref{eq:kl_expansion}.
\State Take an $H_2$-orthonormal basis $Q = \begin{bmatrix} q_1, & \cdots &, q_{k+p}\end{bmatrix}$ for $\range(Y)$. 
\State Apply operator $\mathcal{A}^*$ to $Q$: $\mathcal{A}^* Q = \begin{bmatrix} \mathcal{A}^*q_1, & \cdots &, \mathcal{A}^*q_{k+p} \end{bmatrix}$. 
\State Return $\widehat{\mathcal{A}}:= Q (\mathcal{A}^* Q )^* $ in factored form.
\end{algorithmic}
\end{algorithm}
At this point, this is clearly an \emph{idealized} algorithm, which is not yet practical simply because
sampling according to~\eqref{eq:kl_expansion} appears to require knowledge of the singular vectors and singular values of $\mathcal{A}$.
In~\Cref{section:fdapprox}, we will investigate practical versions of~\Cref{alg:infrsvd} by relating it to the randomized SVD applied to a suitable discretization of $\mathcal{A}$\rev{; we will show that the discretized version approaches the output of \Cref{alg:infrsvd}.} Yet, it is instructive to first establish error estimates for the idealized algorithm.

\subsection{Analysis of~\Cref{alg:infrsvd}}\label{sec:analysis_infrsvd}

We derive stochastic error bounds for~\Cref{alg:infrsvd} by mimicking the proof strategy for the finite-dimensional case presented in~\cite{nakatsukasa2020fast,practicalsketching,Woodruff2014}.  For this purpose, we write the columns of $Y$ as
\begin{equation} \label{eq:defyj}
  y_j = \sum\limits_{i} \omega_{ij} \sigma_{i}  u_i, \quad j = 1,\ldots, k+p,
\end{equation}
and collect the coefficients $\omega_{1j},\ldots, \omega_{kj}$ in the matrix
\begin{equation} \label{eq:BigOmegak}
 \Omega_k = [\omega_{ij}] \in \R^{k \times (k+p)}, \quad i = 1,\ldots,k, \quad j = 1,\ldots, k+p.
\end{equation}
Because of $\omega_{ij} \sim N(0,1)$ (i.i.d.), $\Omega_k$ is a Gaussian random matrix. In particular,
it has full row rank $k$ almost surely, and thus $\Omega_k \Omega_k^\dagger = I_k$, where $\Omega_k^\dagger$ denotes the pseudoinverse of $\Omega_k$. This allows us to state the following structural bound.

\begin{lemma} \label{lemma:structural}
Almost surely (with respect to the Gaussian measure defining $Y$) the approximation $\widehat{\mathcal{A}}:= Q Q^* \mathcal{A}$ returned by~\Cref{alg:infrsvd} satisfies
 \[
  \|\mathcal A - \widehat {\mathcal A}\|_{\HS}^2 \le \sum_{i > k} \sigma_i^2  + \| (\Id - U_k U_k^*) Y \Omega_k^\dagger \|_{\HS}^2,
 \] 
where $U_k = \begin{bmatrix} u_1, & \cdots &, u_k \end{bmatrix}$ is the quasi-matrix containing the first $k$ left singular vectors of $\mathcal A$ and $\Omega_k$ is defined as above. 
\end{lemma}
\begin{proof}
Because $QQ^*: H_2 \to H_2$ is the orthogonal projector onto the range of $Y$, we have $(\Id - QQ^*) Y = 0$ and, hence, we may write
\begin{equation}\label{eq:differenterrorterm}
 \mathcal A - \widehat{\mathcal A} = (\Id - QQ^*) \mathcal A = (\Id - QQ^*) ( \mathcal A - Y \Omega_k^\dagger V_k^*),
\end{equation}
where $V_k$ is the quasi-matrix containing the first $k$ right singular vectors.
From~\eqref{eq:defyj}, it follows that \rev{$U_k^* Y \Omega_k^\dagger = \Sigma_k \Omega_k \Omega_k^\dagger = \Sigma_k$} and thus
\begin{align}
\begin{split}\label{eq:projectioninvariant}
 (\Id - U_k U_k^*)( \mathcal A - Y \Omega_k^\dagger V_k^*) &=  \mathcal A - U_k U_k^* A +  U_k U_k^* Y \Omega_k^\dagger V_k^* - Y \Omega_k^\dagger V_k^*\\
 &= \mathcal A - U_k \Sigma_k V_k^* +  U_k \Sigma_k V_k^* - Y \Omega_k^\dagger V_k^* = \mathcal A - Y \Omega_k^\dagger V_k^*.
 \end{split}
\end{align}
 In summary, the error satisfies
\begin{align*}
  \|(\Id - QQ^*) \mathcal A\|_{\HS}^2 & =\| (\Id - QQ^*)(\Id - U_k U_k^*) ( \mathcal A - Y \Omega_k^\dagger V_k^*) \|_{\HS}^2\\
  &\le 
 \| (\Id - U_k U_k^*) ( \mathcal A - Y \Omega_k^\dagger V_k^*) \|_{\HS}^2 \\
 &= \| (\Id - U_k U_k^*) \mathcal A \|_{\HS}^2 + \| (\Id - U_k U_k^*) Y \Omega_k^\dagger V_k^* \|_{\HS}^2 \\
 &= \sum_{i > k}\sigma_i^2  + \| (\Id - U_k U_k^*) Y \Omega_k^\dagger \|_{\HS}^2,
\end{align*}
where the first equality uses \eqref{eq:differenterrorterm}--\eqref{eq:projectioninvariant} and the second equality uses the orthogonality of the co-ranges of the two involved terms.
\end{proof}

\begin{theorem}\label{thm:InfRandSVD}
If $p \geq 2,$ the approximation $\widehat{\mathcal{A}}:= Q Q^* \mathcal{A}$ returned by~\Cref{alg:infrsvd}
satisfies 
\begin{equation} \label{eq:expectedvalues}
 \mathbb E \|\mathcal A - \widehat {\mathcal A}\|_{\HS}^2 \le \left( 1+ \frac{k}{p-1} \right) \sum_{i>k}\sigma_i^2.
\end{equation}
Furthermore, if $k \geq 2$ and $p \geq 4$, then for all $u,t \geq 1$ the inequality
\begin{equation}\label{eq: tail bound}
    \|\mathcal{A} - \widehat{\mathcal{A}}\|_\HS \leq \Big( \sum_{i>k}\sigma_i^2 \Big)^{1/2} + \eta, \quad \text{with}
    \quad \eta = t \sqrt{\frac{3k}{p+1}} \Big( \sum_{i>k}\sigma_i^2 \Big)^{1/2}
+ ut\frac{e\sqrt{k+p}}{p+1} \sigma_{k+1},
\end{equation}
holds with probability at least $1-2t^{-p} - e^{-u^2/2}$.
\end{theorem}
\begin{proof}
In light of~\Cref{lemma:structural}, it suffices to study the random variable 
\begin{equation*}
    X := \| (\Id - U_k U_k^*) Y \Omega_k^\dagger \|_{\HS}^2 = \sum\limits_{i >k} \|u_i^*Y \Omega_k^{\dagger}\|_2^2=\sum\limits_{i > k} \sigma_i^2 \|\begin{bmatrix} \omega_{i1}, & \cdots &, \omega_{i,k+p} \end{bmatrix} \Omega_k^{\dagger}\|_2^2.
\end{equation*}
Note that $X$ is almost everywhere the pointwise limit for $n\to \infty$ of the random variables
\begin{equation*}
    X_n:=  \sum\limits_{i=k+1}^{n} \sigma_i^2 \|\begin{bmatrix} \omega_{i1}, & \cdots &, \omega_{i,k+p} \end{bmatrix} \Omega_k^{\dagger}\|_2^2,
\end{equation*}
which are measurable because they only involve finitely many $\omega_{ij}$. Following the proof of~\cite[Theorem 10.5]{rsvd}, it holds that
\[
 \mathbb{E}X_n = \sum\limits_{i=k+1}^{n} \sigma_i^2\, \mathbb{E} \|\begin{bmatrix} \omega_{i1}, & \cdots &, \omega_{i,k+p} \end{bmatrix} \Omega_k^{\dagger}\|_2^2
 = \frac{k}{p-1} \sum\limits_{i=k+1}^{n} \sigma_i^2,
\]
where we used that $\Omega_k$ and $\omega_{i1},\ldots, \omega_{i,k+p}$ are independent for $i \ge k+1$. By monotone convergence, 
this establishes $\mathbb E X = \frac{k}{p-1} \cdot (\sigma_{k+1}^2 + \cdots )$. Together with Lemma~\ref{lemma:structural}, this shows~\eqref{eq:expectedvalues}.

To show~\eqref{eq: tail bound}, let us define
\[
\eta_n \coloneqq t \sqrt{\frac{3k}{p+1}} \Big( \sum_{i=k+1}^n \sigma_i^2 \Big)^{1/2} + ut\frac{e\sqrt{k+p}}{p+1} \sigma_{k+1}
\]
and note that $\eta = \lim_{n \to \infty} \eta_n$. In light of~\Cref{lemma:structural}, we need to estimate the probability that $X > \eta^2$. Note that $\eta_n$ converges from below to $\eta$. Following the arguments sketched in the proof of~\cite[Theorem 10.8]{rsvd}, 
we have that
\[
\mathbb P(X_n > \eta^2) \le \mathbb P(X_n > \eta_n^2) \le 2t^{-p} + e^{-u^2/2}.
\]
Since the random variables $X_n$ are nonnegative, monotonically increasing 
, applying the monotone convergence theorem to the random variables $\widehat{X}:=\chi_{\{X > \eta^2\}}$ and $\widehat{X}_n:=\chi_{\{X_n > \eta^2\}}$, where $\chi$ denotes the characteristic function, implies
\[
\mathbb P(X > \eta^2) =\mathbb{E}\widehat{X} =  \lim_{n\to \infty} \mathbb{E}\widehat{X}_n=\lim_{n\to \infty} \mathbb P(X_n > \eta^2) \le 2t^{-p} + e^{-u^2/2}.
\]
By Lemma~\ref{lemma:structural}, this shows that 
$\|\mathcal{A} - \widehat{\mathcal{A}}\|_\HS^2 \le \sum\limits_{i > k} \sigma_i^2 + \eta^2$ 
holds with probability at least $1-2t^{-p} - e^{-u^2/2}$. This implies the more common formulation~\eqref{eq: tail bound}.
\end{proof}

The results of \Cref{thm:InfRandSVD} match the corresponding results in the finite-dimensional case; Theorems 10.5 and 10.7 in~\cite{rsvd}. \rev{To compare with the seminal results by Boullé and Townsend, let us recall that the method from~\cite{boulle2022learning} samples a quasi-matrix $\Omega$ with $k+p$ columns, each independently distributed as $\mathcal{N}_{H_1}(0,\mathcal{K})$
for a given self-adjoint positive covariance operator $\mathcal{K}:H_1 \to H_1$. It then computes an orthonormal basis $Q$ for $\range(\mathcal{A} \Omega)$ and return $QQ^*\mathcal{A}$ as an approximation to $\mathcal{A}$. Error bounds were established in \cite{boulle2022learning,boulleiclr} , and later improved in \cite{perssonboullekressner2025}. In particular, Theorem 3.3 and Remark 2.4 in~\cite{perssonboullekressner2025} imply the estimate
\begin{equation} \label{eq:bt}
    \mathbb{E}\|\mathcal{A} - QQ^* \mathcal{A} \|_{\HS}^2 \leq \left(1 + \frac{k}{p-1} \beta_k + \delta_k\right)\sum\limits_{i> k} \sigma_i^2,
\end{equation}
where the constants $\beta_k$ and $\delta_k$ measure the interaction between the covariance operator $\mathcal{K}$ and the operator $\mathcal{A}$. In particular, as explained in~\cite[Section 2.2]{perssonboullekressner2025}, $\beta_k$ is a measure of the randomness of $\Omega$ in the directions of $V_k^{\bot}$, and $\delta_k$ is a measure of how close $\range(V_k)$ is to an invariant subspace of~$\mathcal{K}$. Ideally, one should choose $\mathcal{K}$ to make $\beta_k$ and $\delta_k$ small. However, without prior information on the operator $\mathcal{A}$, choosing the right covariance operator $\mathcal{K}$ may be difficult. Furthermore, imposing too much regularity on $\mathcal{K}$ (i.e., fast eigenvalue decay) may result in large $\beta_k$ and $\delta_k$. In particular, our numerical experiments will demonstrate how a poorly chosen covariance operator $\mathcal{K}$ may result in a stagnation of the approximation quality. In contrast, the result of~\Cref{thm:InfRandSVD} matches~\eqref{eq:bt} for  the (favorable) constants $\beta_k = 1$ and $\delta_k = 0$, without relying on any prior information.

}

\section{Finite-dimensional approximation}\label{section:fdapprox}

We now aim at relating the infinite-dimensional randomized SVD of the Hilbert--Schmidt operator $\mathcal A : H_1 \to H_2$
to a discretization $A_{m,n}$ of $\mathcal A$. 
For this purpose, we consider a discretization effected by finite-dimensional subspaces
\[
\spann\{w_1,\dots,w_n\} \subset H_1, \qquad \spann\{z_1,\dots,z_m\} \subset H_2,\]
with orthonormal bases $w_1,\dots,w_n$ and $z_1,\dots,z_m$. The compression of 
$\mathcal A$ to these subspaces has the matrix representation 
\begin{equation*}\label{eq: Amn}
A_{m,n} = Z_m^* \mathcal A W_n  = \big[ \langle z_i, \mathcal A w_j \rangle_{H_2} \big]_{i,j = 1}^{m,n} \in \mathbb R^{m \times n},
\end{equation*}
where $Z_m = \begin{bmatrix} z_1, & \cdots &, z_m \end{bmatrix}$ and $W_n = \begin{bmatrix} w_1, & \cdots &, w_n \end{bmatrix}$ are quasi-matrices with orthonormal columns. The compression itself is the linear operator
$\mathcal A_{m,n}: H_1\to H_2$ given by
\begin{equation*} \label{eq:linopamn}
\mathcal A_{m,n} \coloneqq Z_m A_{m,n} W_n^* = Z_m Z_m^* \mathcal A W_n W_n^*.
\end{equation*}

The following simple lemma quantifies the impact on the total error when replacing the discretization $A_{m,n}$ by an approximation $\widehat{A}_{m,n}$, such as the randomized SVD.

\begin{lemma}\label{lemma:discrete}
    For $\widehat{A}_{m,n} \in \mathbb R^{m \times n}$, let $\widehat{\mathcal{A}}_{m,n} \coloneqq Z_m \widehat{A}_{m,n} W_n^* : H_1 \to H_2$. Then 
    \begin{equation*}
        \|\mathcal{A} - \widehat{\mathcal{A}}_{m,n}\|_{\HS}^2 = \|\mathcal{A} - \mathcal A_{m,n}\|_{\HS}^2 + \|A_{m,n} - \widehat{A}_{m,n}\|_\F^2. 
    \end{equation*}
\end{lemma}
\begin{proof}
    Let $\{ e_i \}_i$ be an orthonormal basis of $H_1$ such that $e_1 = w_1,\ldots,e_n = w_n$. Then for $i=1,\dots,n$ the vectors $(\mathcal{A} - \mathcal{A}_{m,n})e_i = (\Id - Z_m Z_m^* )\mathcal{A} e_i$ and $({\mathcal A}_{m,n} - \widehat{\mathcal{A}}_{m,n}) e_i = Z_m (A_{m,n} - \widehat{A}_{m,n}) W_n^* e_i$ are orthogonal to each other. At the same time $({\mathcal A}_{m,n} - \widehat{\mathcal{A}}_{m,n})e_i = 0$ for all $i > n$. From the definition of the Hilbert--Schmidt norm we hence obtain
    \begin{align*}
    \|\mathcal{A} - \mathcal{A}_{m,n} + \mathcal{A}_{m,n} - \widehat{\mathcal{A}}_{m,n}\|_{\HS}^2 &= \|\mathcal{A} - \mathcal{A}_{m,n} \|_\HS^2 + \sum_{i=1}^n \| Z_m (A_{m,n} - \widehat{A}_{m,n}) W_n^* e_i \|_{H_2}^2 \\
    &= \|\mathcal{A} - \mathcal{A}_{m,n} \|_\HS^2 + \sum_{i=1}^n \| (A_{m,n} - \widehat{A}_{m,n}) W_n^* e_i \|^2.
    \end{align*}
    The last sum equals the Frobenius norm of $A_{m,n} - \widehat{A}_{m,n}$, since $W_n^* e_1,\dots,W_n^* e_n$ are the canonical basis vectors in $\mathbb{R}^n$.
\end{proof}

\rev{The quantity $\| \mathcal{A} - \widehat{\mathcal{A}}_{m,n}\|_{\HS}$ is the discretization error and influences all subsequent estimates in this section. Making a suitable choice of discretization subspaces, which ensures a sufficiently fast decay of this error when $m,n \to \infty$, can be highly problem-dependent and is not subject of our considerations. In \emph{theory}, an optimal choice is given by the subspaces spanned by the $m$ resp.~$n$ dominant left resp.~right singular vectors yielding the error $\left( \sum_{\ell > \min(m,n)} \sigma_\ell^2 \right)^{1/2}$, but these subspaces are of course usually not available. In fact, for integral operators with $L^2$ kernels $\kappa(x,y)$ such as~\eqref{eq:integraloperator}, estimating the decay of the $L^2$ discretization errors of $\kappa$ in suitable function systems based on smoothness of $\kappa$ is a classical approach for estimating the actual decay of the singular values themselves. Among the wide range of available results let us only mention that for analytic kernels $\kappa$ on bounded domains one can expect exponential decay of the discretization error when using polynomial approximation~\cite{Little1984,Townsend2015,Trefethen2013}, whereas for kernels with Sobolev regularity one can achieve algebraic decay depending on smoothness exponents when using piecewise polynomial or finite element approximation;~see,~e.g.,~\cite{BirmanSolomjak1977,Schwab2006,Griebel2019}. Additional references and some historical notes regarding the important case of radial basis functions can be found in~\cite{Wathen2015}.
}

\subsection{Approximation error of randomized SVD} \label{sec:approxerror}

Suppose that the randomized SVD is applied to obtain a low-rank approximation $\widehat A_{m,n}$ of the discretization $A_{m,n}$, that is,
\begin{equation*}
\widehat A_{m,n} = Q_m Q_m^* A_{m,n},
\end{equation*}
where $Q_m$ is an orthonormal basis of $A_{m,n} \Omega'_n$ with a Gaussian random matrix $\Omega'_n \in \mathbb{R}^{n \times (k+p)}$. Here we use the notation $\Omega_n'$ in order to distinguish the random matrices for obtaining $\widehat A_{m,n}$ from the ones in~\eqref{eq:BigOmegak} involved in the infinite-dimensional SVD. The following theorem decouples the discretization error from the low-rank approximation error. In particular, the error bounds converge to the ones established in~\Cref{thm:InfRandSVD} for the infinite-dimensional randomized SVD  as the discretization converges to the operator in the Hilbert-Schmidt norm.

\begin{theorem} \label{thm:main}
    Consider a Hilbert-Schmidt operator $\mathcal{A}:H_1 \to H_2$. Given quasi-matrices $W_n: \R^n \to H_1$,  $Z_m: \R^m \to H_2$ with orthonormal columns, let $\mathcal A_{m,n} = Z_m A_{m,n} W_n^*$ with the matrix $A_{m,n} = Z_m^* \mathcal{A} W_n \in \mathbb{R}^{m \times n}$. For a Gaussian random matrix $\Omega_n' \in \mathbb{R}^{n \times (k+p)}$ let $Q_m$ be an orthonormal basis for $\range(A_{m,n} \Omega_n')$, defining the approximation 
    \begin{equation} \label{eq:hatAmn}    
     \widehat{\mathcal{A}}_{m,n} = Z_m Q_m Q_m^* A_{m,n} W_n^*.
    \end{equation}
    Then, if $k \geq 2$, we have
\begin{equation*} 
 \mathbb E \|\mathcal{A} - \widehat{\mathcal{A}}_{m,n} \|_{\HS}^2 \le \|\mathcal{A} - \mathcal A_{m,n} \|_{\HS}^2 + \left( 1+ \frac{k}{p-1} \right) \sum_{i>k}\sigma_i^2,
\end{equation*}
where $\sigma_1 \ge \sigma_2 \ge \cdots$ are the singular values of $\mathcal A$.
Furthermore, if $k \geq 2$ and $p \geq 4$, then for all $u,t \geq 1$ the inequality
\begin{equation*}
 \|\mathcal{A} - \widehat{\mathcal{A}}_{m,n} \|_{\HS} \le \|\mathcal{A} - \mathcal A_{m,n} \|_{\HS}
 + \Big( \sum_{i>k}\sigma_i^2 \Big)^{1/2} + \eta
\end{equation*}
holds with probability at least $1-2t^{-p} - e^{-u^2/2}$ for $\eta$ defined as in~\eqref{eq: tail bound}.
\end{theorem}
\begin{proof}
    By \Cref{lemma:discrete}, 
    $
        \|\mathcal{A} - \widehat{\mathcal{A}}_{m,n} \|_{\HS}^2 =
        \|\mathcal{A} - \mathcal A_{m,n}\|_{\HS}^2 + \|(I-Q_m Q_m^*) A_{m,n}\|_\F^2.
    $
    As the second term is the error of the randomized SVD applied to $A_{m,n}$, the usual error bounds apply (that is, the finite-dimensional analog of \Cref{thm:InfRandSVD}), which imply the statements of the theorem but with the singular values of $\mathcal A$ replaced by the singular values 
    $\hat \sigma_1 \ge \hat \sigma_2 \ge \cdots$ of $A_{m,n}$. The 
    proof is completed by noting that $\hat \sigma_i \le \sigma_i$ because $A_{m,n}$ is a compression of $\mathcal A$ \rev{\cite[Theorem 1.6]{simon}}.
\end{proof}

\begin{remark} \label{remark:accessmodel}
    In some settings, one might only have implicit access to the operator $\mathcal{A}$ through the action on vectors $f\mapsto \mathcal{A}f$ and $g \mapsto \mathcal{A}^*g$. In these cases, it may be undesirable to explicitly form the matrix $A_{m,n}$, which would require $\min\{m,n\}$ applications of $\mathcal{A}$. Instead, one can access $A_{m,n}$ \emph{implicitly} through operator application. For example, applying $\mathcal{A}$ to $W_n \Omega_n'$ and then projecting gives
    \[
    Z_m^*\left[\mathcal{A}(W_n \Omega_n')\right] = A_{m,n}\Omega_n',
    \]
    which requires $k+p$ products with $\mathcal{A}$. If $Q_m$ is an orthonormal basis for $\range(A_{m,n}\Omega_n')$, then $Z_mQ_m$ is an orthonormal basis for $\range(\mathcal{A}_{m,n} W_n \Omega_n')$. We can then form the approximation (in factored form)
    \[
    \widehat{\mathcal{A}}_{m,n} = Z_m Q_m \left[(Q_m^* Z_m^*)\mathcal{A}\right]{W}_n {W}_n^*,
    \]
    which requires $k+p$ additional products with $\mathcal{A}^*$. 
\end{remark}
\begin{remark}
The approximation considered in \Cref{thm:main} differs from that in \cite{boulle2022learning}. There, a low-rank approximation of the operator $\mathcal{A}$ is constructed by first sampling a random quasimatrix $\Psi$ with $k+p$ columns drawn i.i.d.~from a Gaussian measure in $H_1$ with a suitable covariance operator $\mathcal{K}$ in order to then obtain an orthonormal basis $\widehat{Q}$ for  $\range(\mathcal{A}\Psi)$ and return the approximation $\widehat{Q} \widehat{Q}^*\mathcal{A}$. For a specific choice of $\mathcal K$ and basis $Z_m$ our result can be phrased in a similar form. Let $\Psi = W_n \Omega_n'$, then $\Psi$ is a quasi-matrix whose columns are i.i.d.~samples from a Gaussian measure with the  covariance operator $\mathcal{K} = W_nW_n^*$. If $Z_m$ is chosen as an orthonormal basis for $\range(\mathcal{A}W_n)$ we have $\mathcal{A}W_n \Omega_n' = Z_m A_{m,n} \Omega_n'$. Therefore, for this choice of $Z_m$, if $Q_m$ is an orthonormal basis for $\range(A_{m,n}\Omega_n')$, then $\widehat Q = Z_m Q_m$ is indeed an orthonormal basis for $\range(\mathcal{A}W_n \Omega_n') = \range(\mathcal A \Psi)$. Similar to \Cref{lemma:discrete}
we have the following inequality,
\begin{equation}\label{eq:boulletownsend_inequality}
    \|\mathcal{A} - \widehat Q \widehat Q^* \mathcal{A}\|_{\HS} \leq \|\mathcal{A} - \widehat{ \mathcal A}_{m,n}\|_{\HS},
\end{equation}
which relates the approximation error of $\widehat Q\widehat Q^* \mathcal{A}$ to the approximation error of $\widehat{A}_{m,n}$ as obtained in~\eqref{eq:hatAmn}. Hence, all estimates from \Cref{thm:main} also apply to $\|\mathcal{A} - \widehat Q \widehat Q^* \mathcal{A}\|_{\HS}$. 
\end{remark}

\subsection{An adaptive scheme for truncation} \label{sec:adaptive}

When the finite dimensional bases $Z_m$ and $W_n$ are obtained from truncation of complete orthonormal systems $\{w_j\}_{j}$ and $\{z_i\}_{i}$ in the Hilbert spaces $H_1$ and $H_2$, respectively, one needs to choose $m$ and $n$ such that the discretization error will be sufficiently small. This means that $Z_m$ and $W_n$ should resolve the range and co-range of $\mathcal{A}$, respectively. In this section, we develop an adaptive algorithm for this purpose. In particular, given a specified tolerance $\epsilon$ (e.g., set to machine precision), we seek $m$ and~$n$ sufficiently large so that the discretization error satisfies
\begin{equation}\label{eq:discerror}
    \|\mathcal{A} - Z_m A_{m,n} W_n^* \|_{\HS} \lesssim \epsilon \|\mathcal{A}\|_{\HS},
\end{equation}
where $A_{m,n} \in \mathbb{R}^{m \times n}$ is as in~\Cref{lemma:discrete}. 
Following~\Cref{remark:accessmodel}, we assume that the operator is only accessed through
actions of $\mathcal{A}$ and $\mathcal{A}^*$ to elements from the finite-dimensional subspaces defining the discretization.

We will use an inner-outer scheme for choosing $m,n$. In the outer scheme, we 
use a common doubling strategy in adaptive numerical algorithms and set $n_\ell = 2^\ell$ for increasing $\ell = 1,2,\ldots$ until a certain stopping criterion is met. In the inner scheme, we consider $n_\ell$ fixed and choose $m_\ell$ appropriately.

\subsubsection*{Choosing $m_\ell$}

For fixed $n_\ell$, we aim at determining $m_\ell$ so that the projection onto $Z_{m_\ell}$ sufficiently resolves the range of $\mathcal{A}W_{n_\ell}$:
\begin{equation}\label{eq:error}
    \|(\Id - Z_{m_\ell}Z_{m_\ell}^*)\mathcal{A}W_{n_\ell} \|_{\HS} \lesssim \epsilon\|\mathcal{A}\|_{\HS}.
\end{equation}
Let $\Omega_{n_\ell} \in \mathbb{R}^{n_\ell \times (k+p)}$ be a Gaussian random matrix.\footnote{To reduce random fluctuations in our heuristics, the
random matrices are nested when $\ell$ is increased, that is, $\Omega_{n_{\ell+1}}$ contains $\Omega_{n_{\ell}}$ in its first $n_\ell$ rows and the newly added $n_\ell$ rows contain independent~$\mathcal{N}(0,1)$ random variables.}
Because of
    \begin{equation}\label{eq:expectation}
    \frac{1}{k + p} \mathbb{E}\|(\Id - Z_{m_\ell}Z_{m_\ell}^*) \mathcal{A}W_{n_\ell}\Omega_{n_\ell} \|_{\HS}^{2} = \|(\Id - Z_{m_\ell}Z_{m_\ell}^*)\mathcal{A}W_{n_\ell}\|_{\HS}^2,
\end{equation}
we can view $\mathcal{A}W_{n_\ell} \Omega_{n_\ell}$ as a proxy for (the more expensive) $\mathcal{A}W_{n_\ell}$. Inserting this substituion into the criterion~\eqref{eq:error} then requires us to verify whether the tail of the expansion in the basis $\{z_i\}_{i}$ is negligible for every $\mathcal{A}W_{n_\ell} (\Omega_{n_\ell})_j \in H_2$, where $(\Omega_{n_{\ell}})_j$ denotes the $j$th column of $\Omega_{n_{\ell}}$. Because checking these (infinite) tails is computationally infeasible, we verify instead whether the expansion coefficients $z_i^* \mathcal{A}W_{n_\ell} (\Omega_{n_\ell})_j$
for $i = 1,\ldots,m_\ell$ contain entries that are negligible relative to the largest coefficient. In summary, we declare 
$m_\ell$ to be suitably chosen when
\begin{equation} \label{eq:critml}
 \min\limits_{i = 1,\ldots,m_\ell} \|z_i^* \mathcal{A}W_{n_\ell} (\Omega_{n_\ell})_j\|_{\infty}  \leq \epsilon  \max\limits_{i = 1,\ldots,m_\ell} \|z_i^* \mathcal{A}W_{n_\ell} (\Omega_{n_\ell})_j\|_{\infty}  \quad \text{for every } j = 1,\ldots,k+p,
\end{equation}
where $\|\cdot\|_{\infty}$ denotes the $\infty$-norm of a vector. If this condition is not satisfied, 
we double $m_\ell$ and repeat the procedure. 

The heuristic criterion~\eqref{eq:critml} is similar in spirit to the one used by Chebfun for truncating expansions; see \cite[Sec.~2]{battlestrefethen}.  In our numerical experiments $\mathcal{A}$ is an operator $L^2([-1,1]^d) \to L^2([-1,1]^d)$, for $d = 1$ or $2$, and, in fact, we delegate the choice of $m_\ell$ to Chebfun. Although Chebfun uses Chebyshev interpolation/expansion rather than $L^2$-orthogonal projection, Chebfun ensures that the approximations are accurate up to machine precision. Thus, for practical purposes, Chebfun's approximation can be regarded as numerically equivalent to the projection-based procedure explained above.

\subsubsection*{Choosing $\ell$} 

It remains to discuss the choice of the refinement level $\ell$ such that the approximation error bound~\eqref{eq:discerror} is satisfied when setting $n \equiv n_\ell = 2^\ell$ and $m \equiv m_\ell$ (chosen as explained above).
As a simple heuristic criterion, we stop increasing $\ell$ when the approximation error is not sufficiently reduced anymore:
\begin{equation*}
    \|\mathcal{A}- Z_{m_\ell}Z_{m_\ell}^* \mathcal{A} W_{n_\ell}W_{n_\ell}^*\|_{\HS}^2 - \|\mathcal{A}- Z_{m_{\ell-1}}Z_{m_{\ell-1}}^* \mathcal{A} W_{n_{\ell-1}}W_{n_{\ell-1}}^*\|_{\HS}^2 \lesssim \epsilon^2 \|\mathcal{A}\|_{\HS}^2.
\end{equation*}
Using Pythagoras theorem and the orthonormality of the systems $\{z_i\}_{i}$ and $\{w_j\}_{j}$ we obtain
\begin{align*}
    &\|\mathcal{A}- Z_{m_\ell}Z_{m_\ell}^* \mathcal{A} W_{n_\ell}W_{n_\ell}^*\|_{\HS}^2 - \|\mathcal{A}- Z_{m_{\ell-1}}Z_{m_{\ell-1}}^* \mathcal{A} W_{n_{\ell-1}}W_{n_{\ell-1}}^*\|_{\HS}^2 \\
    {}={}&\|Z_{m_\ell}Z_{m_\ell}^*\mathcal{A} W_{n_\ell}W_{n_\ell}^* - Z_{m_{\ell-1}}Z_{m_{\ell-1}}^*\mathcal{A} W_{n_{\ell-1}}W_{n_{\ell-1}}^*\|_{\HS}^2.
\end{align*}
Recalling~\eqref{eq:expectation} and using $\|Z_{m_\ell}Z_{m_\ell}^*\mathcal{A} W_{n_\ell}W_{n_\ell}^*\|_{\HS}$ as a proxy for $\|\mathcal{A}\|_{\HS}$ we stop the scheme when
    \begin{equation*}
    \|Z_{m_\ell}Z_{m_\ell}^*\mathcal{A} W_{n_\ell}\Omega_{n_\ell} - Z_{m_{\ell-1}}Z_{m_{\ell-1}}^*\mathcal{A} W_{m_{\ell-1}}\Omega_{n_{\ell-1}}\|_{\HS} \leq \epsilon \|Z_{m_\ell}Z_{m_\ell}^*\mathcal{A} W_{m_\ell}\Omega_{n_\ell}\|_{\HS}.
\end{equation*}
We then set $m = m_\ell$, $n = n_\ell$, and reuse the sketch $Z_{m}Z_{m}^*\mathcal{A} W_{n}\Omega_{n}$ when carrying out the randomized  SVD. 

\subsection{Convergence of randomized SVD} \label{sec:convrandSVD}

\Cref{thm:main} already indicates that the idealized randomized SVD of the original operator $\mathcal{A}$ in \Cref{alg:infrsvd} and the randomized SVD of its discretization $A_{m,n}$, as obtained in \eqref{eq:hatAmn}, are closely related. The purpose of this section is to explore this relation in more detail and study the convergence of the
randomized SVD \emph{itself} as $m,n\to \infty$. The discrete approximation $\widehat{\mathcal{A}}_{m,n}$ in~\eqref{eq:hatAmn} involves a random matrix of the form $A_{m,n}\Omega'_n = Z_m^* \mathcal{A} W_n \Omega'_n$ for a Gaussian random matrix $\Omega'_n \in \R^{n \times (k+p)}$. The columns of the quasi-matrix
\begin{equation}\label{eq:Yn}
Y_n = \mathcal{A} W_n \Omega'_n
\end{equation}
are independent and distributed according to the Gaussian measure
\[
 \mu_n \coloneqq \mathcal{N}_{H_2}(0,\mathcal{A}W_n W_n^* \mathcal{A}^*)
\]
on $H_2$. On the other hand, the columns of the quasi-matrix $Y$ involved in the idealized infinite-dimensional randomized SVD of $\mathcal A$ (\Cref{alg:infrsvd})
are independent and distributed according to 
\[
 \mu \coloneqq \mathcal{N}_{H_2}(0,\mathcal{A}\mathcal{A}^*).
\]
In order to obtain reasonable convergence statements, the measures $\mu$ and $\mu_n$ need to be coupled. More specifically, we will consider the set $\Gamma(\mu_n,\mu)$ of all joint probability Borel measures on $H_2 \times H_2$ with 
marginals $\mu_n$ and $\mu$. For the columns of $Y$ and $Y_n$, choosing the optimal coupling leads to the notion of 
the Wasserstein-$2$ distance between $\mu_n$ and $\mu$ defined as
\begin{equation} \label{eq:defwasserstein}
 W_2(\mu_n, \mu):= \Big( \inf_{\gamma \in \Gamma(\mu_n,\mu)} \mathbb E_{(y_n,y) \sim \gamma}  \| y - y_n \|^2_{\HS} \Big)^{1/2}.
\end{equation}
The following lemma shows that $\mu_n$ approaches $\mu$ as the discretization of $\mathcal{A}$ converges.
\begin{lemma} \label{lemma:wassersteinnew}
The Wasserstein-$2$ distance~\eqref{eq:defwasserstein}
satisfies
$W_2(\mu_n, \mu) \leq \|\mathcal{A} - \mathcal{A}W_nW_n^*\|_{\HS}.$
\end{lemma}
\begin{proof}
Theorem~3.5 in~\cite{gelbrich1990on} implies that $W_2(\mu_n, \mu)^2$ admits the explicit expression
\[
  \Tr(\mathcal{A}\mathcal{A}^*) + \Tr(\mathcal{A}W_nW_n^*\mathcal{A}^*) 
        -2\Tr\!\left(\big((\mathcal{A}W_nW_n^*\mathcal{A}^*)^{1/2}\mathcal{A}\mathcal{A}^*(\mathcal{A}W_nW_n^*\mathcal{A}^*)^{1/2}\big)^{1/2}\right).
\]
Noting that $\mathcal{A}\mathcal{A}^* \succeq \mathcal{A}W_nW_n^*\mathcal{A}^*\succeq 0$, we can upper bound the last term in this expression by $-2\Tr(\mathcal{A}W_nW_n^*\mathcal{A}^*)$. In summary,
\[
 W_2(\mu_n, \mu)^2 \le \Tr(\mathcal{A}\mathcal{A}^*) - \Tr(\mathcal{A}W_nW_n^*\mathcal{A}^*) = \|\mathcal{A} - \mathcal{A}W_nW_n^*\|_{\HS}^2,
\]
    as desired.
\end{proof}

\Cref{lemma:wassersteinnew} can be used to derive convergence statements on the randomized SVD applied 
to the discretization $A_{m,n}$ as $m,n\to \infty$.  In particular, the following theorem shows convergence in expectation (under optimal couplings).
The expected value is taken with respect to the joint distribution
$\gamma^{\otimes(k+p)}$ of the $k+p$ independent column pairs of $(Y,Y_n)$. 

\begin{theorem} \label{thm:convergenceexpect}
Consider a Hilbert-Schmidt operator $\mathcal A: H_1 \to H_2$ 
of rank at least $k+p+2$. 
Suppose that the compression
${\mathcal{A}}_{m,n} = Z_m A_{m,n} W_n^*$, defined as above, satisfies $\| {\mathcal{A}} - {\mathcal{A}}_{m,n} \|_{\HS} \to 0$ as $m,n\to \infty$. Then the approximations $\widehat{\mathcal{A}}$ and $\widehat{\mathcal{A}}_{m,n}$ 
produced by applying the randomized SVD with $k+p$ samples to 
$\mathcal A$ and its discretization $A_{m,n}$, respectively, satisfy
\[
 \inf_{\gamma \in \Gamma(\mu_n,\mu)} \mathbb E \| \widehat{\mathcal{A}} - \widehat{\mathcal{A}}_{m,n} \|_{\HS} \le C 
 \| {\mathcal{A}} - {\mathcal{A}}_{m,n} \|_{\HS},
\]
for some constant $C$ (independent of $m,n$), where the expectation is taken with respect to $\gamma^{\otimes(k+p)}$. In particular, $\lim\limits_{m,n \to \infty} \left[\inf_{\gamma \in \Gamma(\mu_n,\mu)} \mathbb E \| \widehat{\mathcal{A}} - \widehat{\mathcal{A}}_{m,n} \|_{\HS}\right] = 0$.
\end{theorem}
\begin{proof}
To simplify notation, we set $\varepsilon \coloneqq \| {\mathcal{A}} - {\mathcal{A}}_{m,n} \|_\HS$ and $r = k+p$. Additional to the (random) matrix $Y$ of \Cref{alg:infrsvd}, we will make use of
$Y_n$ in~\eqref{eq:Yn} and $Y_{m,n} = Z_m Z_m^* Y_n = Z_m A_{m,n} \Omega'_n$. 
In view of~\eqref{eq:hatAmn}, this allows us to express 
\[
 \widehat{\mathcal{A}}_{m,n} = Z_m P_{A_{m,n} \Omega'_n} A_{m,n} W_n^* = P_{Y_{m,n}} \mathcal{A} W_n W_n^*,
\]
where $P_{X}$ denotes the orthogonal projection onto the range of a (quasi-)matrix $X$.
Therefore,
\begin{align}
 \| \widehat{\mathcal{A}} - \widehat{\mathcal{A}}_{m,n} \|_{\HS} &= 
 \| P_Y \mathcal{A} - P_{Y_{m,n}} \mathcal{A} W_n W_n^*  \|_\HS \nonumber \\
 &\le  \| (P_Y - P_{Y_{m,n}}) \mathcal{A}  \|_{\HS} + \|P_{Y_{m,n}} \mathcal{A} (\mathrm{Id} - W_n W_n^*)  \|_{\HS} \nonumber \\
 &\le  \| (P_Y - P_{Y_{m,n}}) \mathcal{A}  \|_{\HS} + \varepsilon \nonumber \\
 &\le  \| (P_Y - P_{Y_{n}}) \mathcal{A}  \|_{\HS} + \| ( P_{Y_n} - P_{Y_{m,n}}) \mathcal{A}  \|_{\HS} +\varepsilon.\label{eq:blubber}
\end{align}
It remains to bound the first two terms of~\eqref{eq:blubber}.
By the assumptions, the $(r+2)$-nd singular value of $\mathcal A$ is positive, $\sigma \coloneqq \sigma_{r+2}(\mathcal A)>0$, 
and hence $\sigma_{r+2}(A_{m,n}) = \sigma_{r+2}(\mathcal A_{m,n}) > \sigma/2$ for $m,n$ sufficiently large. From their definitions, it then follows that  the ranges of $Y, Y_n, Y_{m,n}$ all have dimension $\rev{r}$ almost surely. In particular,
\[
 \|Y^\dagger\|_{\op} = \sigma_{r}(Y)^{-1} \le \sigma_{r}(U^*_{r+2}Y)^{-1} = 
 \sigma_{r}(\Sigma_{r+2} \Omega_{r+2})^{-1} \le \|\Omega_{r+2}^\dagger\|_2 / \sigma,
\]
where we used the definition~\eqref{eq:kl_expansion} of the columns of $Y$, and the (quasi-)matrices 
$U_{r+2}$ and $\Sigma_{r+2}$ contain the first $r+2$ left singular vectors and singular values of $\mathcal A$.
The Gaussian random matrix $\Omega_{r+2} \in \R^{(r+2)\times r}$ collects the random expansion coefficients,
as in~\eqref{eq:BigOmegak}. By~\cite[Lemma 3.1]{nakatsukasa2020fast}, $\mathbb E \|\Omega_{r+2}^\dagger\|^2_{\op}$ is bounded by a constant (only depending on $r$). An analogous argument applies to $Y_{m,n}$ and, thus, there is a constant
$c$ (depending only on $r$ and $\sigma$) such that
\begin{equation} \label{eq:constants}
\mathbb E \|Y^\dagger\|^2_{\op} \le c,\qquad \mathbb E \|Y_{m,n}^\dagger\|^2_{\op} \le c.
\end{equation}
\Cref{lemma:projpert} and the Cauchy-Schwarz inequality imply that
\[
    \mathbb{E}\|P_{Y} - P_{Y_n}\|_{\op}  \leq \mathbb{E}\left[\|Y^{\dagger}\|_{\op}\|Y-Y_n\|_{\op}\right]
    \leq \sqrt{\mathbb{E}\|Y^{\dagger}\|_{\op}^2}\sqrt{\mathbb{E}\|Y-Y_n\|_{\op}^2},
\]
where the expectation is taken with respect to $\gamma^{\otimes (k+p)}$ for any $\gamma \in \Gamma(\mu_n,\mu)$.
Now,
\begin{equation*} 
 \inf_{\gamma \in \Gamma(\mu_n,\mu)} \mathbb{E}\|Y-Y_n\|^2_{\op} \le 
 \inf_{\gamma \in \Gamma(\mu_n,\mu)} \mathbb{E}\|Y-Y_n\|^2_{\HS} = \sum_{j = 1}^r \inf_{\gamma \in \Gamma(\mu_n,\mu)} \mathbb{E} \|y_j - y_{n,j} \|_2^2 \le r \varepsilon^2,
\end{equation*}
where the last step follows from Lemma~\ref{lemma:wassersteinnew}. Combining the inequalities above gives
\[
 \inf_{\gamma \in \Gamma(\mu_n,\mu)} \mathbb E \| (P_Y - P_{Y_{n}}) \mathcal{A}  \|_{\HS} \le 
 \| \mathcal{A}  \|_{\HS}\cdot \inf_{\gamma \in \Gamma(\mu_n,\mu)} \mathbb E \| P_Y - P_{Y_{n}}\|_{\op}
 \le \sqrt{c r}\| \mathcal{A}  \|_{\HS} \varepsilon,
\]
establishing the desired bound for the first term in~\eqref{eq:blubber}. The second term is bounded in an analogous fashion. Using $\mathbb{E}\|P_{Y_n} - P_{Y_{m,n}}\|_{\op}  \leq \sqrt{\mathbb{E}\|Y_{m,n}^{\dagger}\|_{\op}^2}\sqrt{\mathbb{E}\|Y_n - Y_{m,n}\|_{\op}^2}$,~\eqref{eq:constants}, and 
\begin{eqnarray}
 \mathbb E \|Y_n - Y_{m,n} \|^2_{\op} & = & \mathbb E \| (\mathrm{Id}-Z_mZ_m^*) \mathcal  A W_n \Omega_n' \|^2_{\op} 
  \le  \mathbb E \| (\mathrm{Id}-Z_mZ_m^*) \mathcal A W_n \Omega_n' \|^2_{\HS} \nonumber \\
 &=& r \| (\mathrm{Id}-Z_mZ_m^*) \mathcal  A W_n \|^2_{\HS} \le r \varepsilon^2 \label{eq:hsid}
\end{eqnarray}
where the second equality in~\eqref{eq:hsid} is standard~\cite[Proposition 10.1]{rsvd},
we arrive at the estimate
$
\mathbb E \| ( P_{Y_n} - P_{Y_{m,n}}) \mathcal{A}  \|_{\HS} \le \sqrt{c r}\| \mathcal{A}  \|_{\HS} \varepsilon.
$
In summary, the sum~\eqref{eq:blubber} is bounded by some constant times $\varepsilon$ when taking the 
infimum with respect to $\gamma$.
\end{proof}

\section{Nyström approximation in Hilbert spaces}\label{section:nystrom}

Given a symmetric positive semi-definite (SPSD) matrix $T_n \in \mathbb{R}^{n \times n}$, the basic form of the Nyström approximation \cite{gittensmahoney, tropp2023randomized,perssonboullekressner2025} proceeds by sampling a random Gaussian matrix $\Omega_{n} \in \mathbb{R}^{n \times r}$ and forming the rank-$r$ approximation
\begin{equation}\label{eq:nystrom}
    T_n \approx T_n \Omega_n (\Omega_n^* T_n \Omega_n)^{\dagger}(T_n \Omega_n)^*,
\end{equation}
where $\dagger$ denotes the Moore-Penrose pseudoinverse. Compared to the randomized SVD, this approximation bypasses the need for a second multiplication of $T_n$ (with $Q$) and preserves the SPSD property. The purpose of this section is to discuss a suitable infinite-dimensional extension of~\eqref{eq:nystrom}. While~\cite{perssonboullekressner2025} parallels the developments of~\cite{boulle2022learning}, which replaces $\Omega_n$ by random elements from a Gaussian process, we will develop and analyze an extension in the spirit of the infinite-dimensional randomized SVD of~\Cref{alg:infrsvd}.

It is instructive to recall a well-known connection between the Nyström approximation and the randomized SVD; see, e.g.,~\cite[Lemma 1]{gittensmahoney} and \cite[Proof of Proposition 1]{gittens}. For this purpose, we consider the alternative representation
\begin{equation}\label{eq:gram}
    T_n \Omega_n (\Omega_n^* T_n \Omega_n)^{\dagger}(T_n \Omega_n)^* = T_n^{1/2}P_{T_n^{1/2} \Omega_n} T_n^{1/2},
\end{equation}
where $T_n^{1/2}$ denotes the principal square-root of $T_n$ and $P_{T_n^{1/2} \Omega_n}$ denotes the orthogonal projection onto $\range(T_n^{1/2} \Omega_n)$. 
Because the error is SPSD, it follows that its nuclear norm coincides with its trace and satisfies
\begin{equation*}
    \|T_n - T_n \Omega_n (\Omega_n^* T_n \Omega_n)^{\dagger}(T_n \Omega_n)^*\|_{\T} = \|T_n^{1/2} - P_{T_n^{1/2} \Omega_n} T_n^{1/2}\|_{\F}^2.
\end{equation*}
Now, the latter expression is precisely the squared Frobenius norm error of the randomized SVD applied to $T_n^{1/2}$. Consequently, error bounds for the randomized SVD can be immediately turned into error bounds for the Nyström approximation.

\subsection{Algorithm}\label{section:infnystrom}

The representation~\eqref{eq:gram} of the Nystr\"om approximation is a suitable starting point for its extension to
a SPSD trace-class operator $\mathcal{T}:H \to H$ on a separable Hilbert space $(H,\langle\cdot,\cdot\rangle_H)$. 
Let $Y$ be a quasi-matrix with $r$ columns drawn independently from $\mathcal{N}_{H}(0,\mathcal{T})$. Then~\eqref{eq:gram} suggests to define the infinite-dimensional Nyström approximation as
\begin{equation}\label{eq:infnystrom}
    \mathcal{T} \approx \mathcal{T}^{1/2} P_{Y} \mathcal{T}^{1/2}.
\end{equation}
Noting that this approximation 
can be equivalently written as $\mathcal{T}^{1/2} Q(\mathcal{T}^{1/2} Q)^*$
for 
an orthonormal basis  $Q$ of $\range(Y)$, we arrive at \Cref{alg:nystrom}.
\begin{algorithm}[t]
\small
\caption{Infinite-dimensional Nyström approximation}
\label{alg:nystrom}
\begin{algorithmic}[1]
\Input
SPSD trace class operator $\mathcal{T}:H \to H$, target rank $k$, oversampling parameter $p$.
\Output
Rank $k+p$ approximation $\widehat{\mathcal{T}}:= \mathcal{T}^{1/2}Q (\mathcal{T}^{1/2} Q)^*$ to $\mathcal{T}$ in factored form.
\State Sample a random quasi-matrix $Y = \begin{bmatrix} y_1, & \cdots &, y_{k+p} \end{bmatrix}$ with $y_j \sim \mathcal{N}_{H}(0,\mathcal{T})$ i.i.d.
\State Take an $H$-orthonormal basis $Q = \begin{bmatrix} q_1, & \cdots &, q_{k+p}\end{bmatrix}$ for $\range(Y)$. 
\State Apply operator $\mathcal{T}^{1/2}$ to $Q$: $\mathcal{T}^{1/2} Q = \begin{bmatrix} \mathcal{T}^{1/2}q_1, & \cdots &, \mathcal{T}^{1/2}q_{k+p} \end{bmatrix}$. 
\State Return $\widehat{\mathcal{T}}:= \mathcal{T}^{1/2}Q (\mathcal{T}^{1/2} Q)^*$ in factored form.
\end{algorithmic}
\end{algorithm}
Similar to the infinite-dimensional randomized SVD, this is clearly an \emph{idealized} algorithm.
In~\Cref{section:finite_nystrom}, we will propose a practical version for a discretization of $\mathcal{T}$, which resembles the original representation~\eqref{eq:nystrom} and, in particular, does not require access to the square root of the operator or its discretization.

As in the case of the randomized SVD, \Cref{alg:nystrom} satisfies error bounds that match the finite-dimensional bounds~\cite{gittensmahoney}.

\begin{theorem}
    Let $\{\sigma_i\}_i$ be the singular values of $\mathcal{T}$. If $p \geq 2$, the approximation $\widehat{\mathcal{T}} = \mathcal{T}^{1/2}Q (\mathcal{T}^{1/2} Q)^*$ returned by \Cref{alg:nystrom} satisfies
    \begin{equation*}
        \mathbb{E}\|\mathcal{T} - \widehat{\mathcal{T}}\|_{\T} \leq \left(1 + \frac{k}{p-1}\right) \sum\limits_{i > k} \sigma_i.
    \end{equation*}
    Furthermore, if $k \geq 2$ and $p \geq 4$, then for all $u,t, \geq 1$ the inequality
    \begin{equation}
        \mathbb{E}\|\mathcal{T} - \widehat{\mathcal{T}}\|_{\T} \leq \sum\limits_{i > k} \sigma_i + \widehat{\eta}^2, \quad \text{with}
    \quad \widehat{\eta} = t \sqrt{\frac{3k}{k+1}} \Big(\sum\limits_{i > k} \sigma_i\Big)^{1/2} + ut \frac{e\sqrt{k+p}}{p+1} \sigma_{k+1}^{1/2},\label{eq:nystromgamma}
    \end{equation}
    holds with probability at least $1-2t^{-p} - e^{-u^2/2}$.
\end{theorem}
\begin{proof}
    Noting $\|\mathcal{T} - \widehat{\mathcal{T}}\|_{\T} = \|\mathcal{T}^{1/2} - QQ^* \mathcal{T}^{1/2}\|_{\HS}^2$ and following the proof of \Cref{thm:InfRandSVD} with $\mathcal{A} = \mathcal{T}^{1/2}$ yields the desired result.
\end{proof}

\subsection{Finite-dimensional approximation}\label{section:finite_nystrom}

We now aim at presenting a practical Nyström approximation by replacing the idealized samples from $\mathcal{N}_H(0,\mathcal{T})$ in \Cref{alg:nystrom} with samples from an approximating Gaussian measure $\mathcal{N}_H(0,\mathcal{T}_n)$, which is supported on a finite-dimensional subspace of $H$. Consider an orthonormal system $\{w_i\}_i$ in $H$ and set $W_n = \begin{bmatrix} w_1, & \cdots &, w_n\end{bmatrix}$. According to~\eqref{eq:infnystrom}, we define a discretized Nyström approximation as
\begin{equation*}
    \widehat{\mathcal{T}}_n := \mathcal{T}^{1/2}P_{Y_n} \mathcal{T}^{1/2},
\end{equation*}
where $Y_n$ is a quasi-matrix with $k$ columns drawn independently from $\mathcal{N}_{H}(0,\mathcal{T}^{1/2}W_nW_n^* \mathcal{T}^{1/2})$.
Rewriting $Y_n = \mathcal{T}^{1/2} W_n \Omega_n$ for an $n \times r$ Gaussian random matrix
$\Omega_n$, one obtains the equivalent representation
\begin{equation}\label{eq:discrete_nystrom}
    \widehat{\mathcal{T}}_n =\mathcal{T} W_n \Omega_n (\Omega_n^* W_n^*\mathcal{T}W_n \Omega_n)^{\dagger} (\mathcal{T} W_n \Omega_n)^*,
\end{equation}
which closely resembles the finite-dimensional Nyström approximation \eqref{eq:nystrom}. In practice, following the developments from \Cref{sec:adaptive},
this suggests to ``evaluate'' each column of $\mathcal{T} W_n \Omega_n$ by expanding it into a complete(d) orthonormal system $\{w_i\}_i$ of $H$ up to machine precision. 

Using the relationship between the randomized SVD 
\begin{equation*}
    \|\mathcal{T} - \widehat{\mathcal{T}}_n\|_{\T} = \|\mathcal{T}^{1/2} - P_{\mathcal{T}^{1/2} W_n \Omega_n} \mathcal{T}^{1/2}\|_{\HS}^2,
\end{equation*}
together with \eqref{eq:boulletownsend_inequality} and $\|\mathcal{T}^{1/2} - W_nW_n^*\mathcal{T}^{1/2}\|_{\HS}^2 = \Tr(\mathcal{T}) - \Tr(W_n^* \mathcal{T} W_n)$, and applying \Cref{thm:main} with $\mathcal{T}^{1/2} = \mathcal{A}$ yields the following result.
\begin{theorem}\label{theorem:discrete_nystrom}
    Consider an SPSD trace-class operator $\mathcal{T}:H \to H$ with singular values $\{\sigma_i\}_i$. Let $W_n :\mathbb{R}^n \to H$ be a quasi-matrix with orthonormal columns. For a Gaussian random matrix $\Omega_n \in \mathbb{R}^{n \times (k+p)}$, define the approximation
    \begin{equation*}
        \widehat{\mathcal{T}}_n = \mathcal{T}\rev{W_n} \Omega_n (\Omega_n^* W_n^* \mathcal{T}W_n \Omega_n)^{\dagger} (\mathcal{T}\rev{W_n} \Omega_n)^*.
    \end{equation*}
    Then, if $k \geq 2$ we have
\begin{equation*} 
 \mathbb E \|\mathcal{T} - \widehat{\mathcal{T}}_n \|_{\T} \le \Tr(\mathcal{T}) - \Tr(W_n^*\mathcal{T}W_n) + \left( 1+ \frac{k}{p-1} \right) \sum_{i>k}\sigma_i.
\end{equation*}
Furthermore, if $k \geq 2$ and $p \geq 4$, then for all $u,t \geq 1$ the inequality
\begin{equation*}
 \|\mathcal{T} - \widehat{\mathcal{T}}_{n} \|_{\T} \le \Tr(\mathcal{T}) - \Tr(W_n^* \mathcal{T} W_n)
 + \sum_{i>k}\sigma_i + \widehat{\eta}^2
\end{equation*}
holds with probability at least $1-2t^{-p} - e^{-u^2/2}$ for $\widehat{\eta}$ defined as in~\eqref{eq:nystromgamma}.
\end{theorem}
\begin{remark}
    If $\{w_i\}_{i}$ forms a complete orthonormal system for $H$, then by \eqref{eq:tracedef} $\Tr(W_n^*\mathcal{T}W_n) \to \Tr(\mathcal{T})$ as $n \to \infty$.
\end{remark}
Furthermore, \Cref{thm:convergenceexpect} naturally generalizes to the Nyström approximation. 
Noting that
\begin{equation*}
    \widehat{\mathcal{T}} = \mathcal{T}^{1/2} \mathcal{P}_{Y} \mathcal{T}^{1/2},
\end{equation*}
where $Y$ is a quasi-matrix whose columns are i.i.d. $\mathcal{N}_{H}(0,\mathcal{T})$. Hence,
\begin{equation*}
    \|\widehat{\mathcal{T}} - \widehat{\mathcal{T}}_n\|_{\HS} = \|\mathcal{T}^{1/2}(P_{Y} - P_{\mathcal{T}^{1/2}W_n \Omega_n}) \mathcal{T}^{1/2}\|_{\HS} \leq \|\mathcal{T}\|_{\T}\|P_{Y} - P_{\mathcal{T}^{1/2}W_n \Omega_n}\|_{\op}.
\end{equation*}
Following the proof of \Cref{thm:convergenceexpect} one can show that $\inf\limits_{\gamma \in \Gamma(\mu_n,\mu)} \mathbb{E}\|P_{Y} -P_{\mathcal{T}^{1/2}W_n \Omega_n}\|_{\op} \to 0$ as $n \to \infty$, where $\mu:=\mathcal{N}_H(0,\mathcal{T})$ and $\mu_n:=\mathcal{N}_H(0,\mathcal{T}^{1/2}W_nW_n^* \mathcal{T}^{1/2})$. In summary, we have the following result. 

\begin{theorem}\label{theorem:nystromconvergence}
Consider the setting of \Cref{theorem:discrete_nystrom} and assume that $\mathcal{T}$ has rank at least $k+p+2$. Then the approximations $\widehat{\mathcal{T}}$ and $\widehat{\mathcal{T}}_n$ 
produced by applying the Nyström approximation~\eqref{eq:infnystrom} with $k+p$ samples and its discrete version~\eqref{eq:discrete_nystrom}, respectively, satisfy
\[
 \inf_{\gamma \in \Gamma(\mu_n,\mu)} \mathbb E \| \widehat{\mathcal{T}} - \widehat{\mathcal{T}}_{n} \|_{\HS} \le C 
 \left[\Tr(\mathcal{T}) -\Tr(W_n^*\mathcal{T}W_n)\right],
\]
for some constant $C$ (independent of $n$), where the expectation is taken with respect to $\gamma^{\otimes(k+p)}$. If $\{w_i\}_i$ forms a complete orthonormal system of $H$ we have $\lim\limits_{n \to \infty}\left[\inf_{\gamma \in \Gamma(\mu_n,\mu)} \mathbb E \| \widehat{\mathcal{T}} - \widehat{\mathcal{T}}_{n} \|_{\HS}\right] =0.$
\end{theorem}

\section{Numerical experiments} \label{sec:numexp}
In this section, we provide some numerical validatation of the infinite-dimensional randomized SVD and Nyström approximation presented in this work. We carry out all operations using the Chebfun software package \cite{chebfun}. 
Two examples are considered. In the first example, we use the randomized SVD to compute a low-rank approximation to the integral operator
\[
\mathcal{A} : L^2([-1,1]) \to L^2([-1,1]), \quad [\mathcal{A} f](x) = \int_{-1}^1 \kappa(x,y) f(y) \mathrm{d} y,
\]
with the kernel defined by \cite{pretty}
\begin{equation}\label{eq:airy}
    \kappa(x,y) = \mathrm{Ai}(5(x+y^2)) \mathrm{Ai}(-5(x^2+y^2)),
\end{equation}
where $\mathrm{Ai}(t) = \frac{1}{\pi} \int_0^{\infty}\cos\left(\frac{s^3}{3} + st\right)\mathrm{d}s$ denotes the Airy function. The bases defining $Z_m$ and $W_n$ are chosen as orthonormalized Legendre polynomials. We run the discrete randomized SVD with three different strategies: with fixed value $n=10$, with fixed value $n=15$, and with $n$ chosen adaptively as outlined in \Cref{sec:adaptive}. In all cases, the choice of $m$ is left to the internal Chebfun routine. We compare these approaches with the idealized version of the randomized SVD (\Cref{alg:infrsvd}) based on the ``exact'' SVD of $\mathcal A$ as computed by Chebfun. The results are presented in \Cref{fig:toy_example} and confirm our theoretical prediction that, provided $m$ and $n$ are sufficiently large, the error from the discrete randomized SVD is nearly indistinguishable from the idealized randomized SVD. In particular, the adaptive scheme performs as well as the idealized algorithm. 

\begin{figure}[htbp]
    \centering
    \begin{overpic}[width=.9\textwidth]{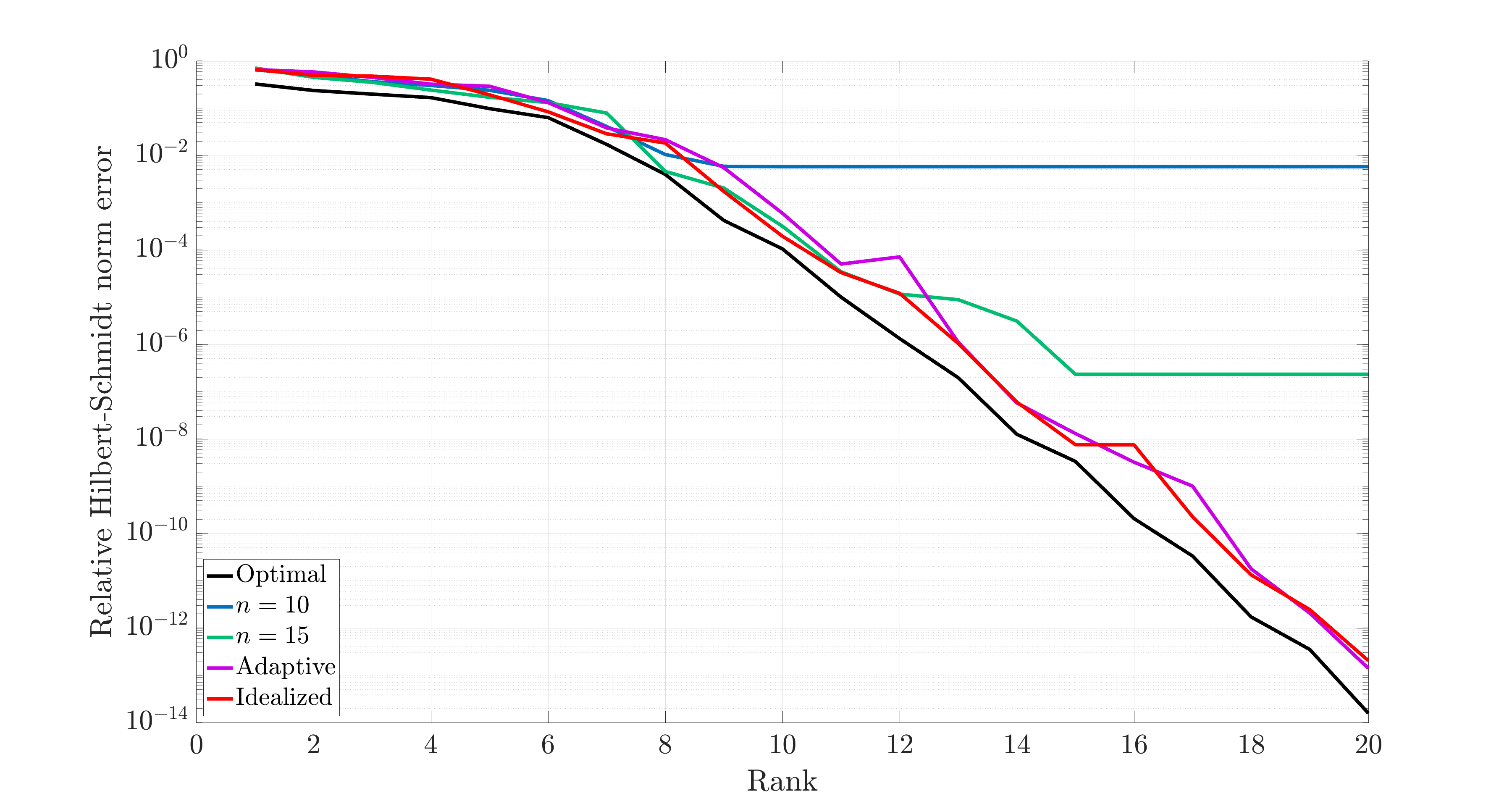}
    \end{overpic}
    \caption{Relative Hilbert--Schmidt norm errors for approximations of the integral operator defined by \eqref{eq:airy} using Legendre polynomials. In the plots for $n=10, 15$ the dimension of the co-range basis $W_n$ is fixed in advance, whereas in \emph{Adaptive} this dimension is chosen adaptively as outlined in \Cref{sec:adaptive}. \emph{Idealized} refers to the idealized algorithm outlined in \Cref{alg:infrsvd}. \emph{Optimal} denotes the optimal low-rank approximation error given by the truncated SVD.
    }
    \label{fig:toy_example}
\end{figure}

As a second experiment we consider computing a Nyström approximation of a Gaussian kernel defined on $[-1,1]^2$:
\begin{equation}\label{eq:sekernel}
    \kappa(x,y) = \exp\left(-\frac{\|x-y\|_2^2}{0.32}\right), \quad x,y \in [-1,1]^2.
\end{equation}
As discussed in \cite[Sec.~4]{perssonboullekressner2025}, obtaining an accurate Nyström approximation of such a kernel is advantageous for efficiently sampling from the associated Gaussian process. A low-rank approximation enables fast sampling, and if the approximation is accurate in the trace norm, the resulting error in the Wasserstein distance between the exact and approximate Gaussian measures is also small; see~\cite[Sec.~4]{perssonboullekressner2025}. In our experiments, we let $n = 25, 625$ and let the basis $W_n$ be a tensor product of the first $\sqrt{n}$ Legendre polynomials. We compare the method with the idealized Nyström approximation from~\eqref{eq:infnystrom}. The results are presented in \Cref{fig:gp}

\begin{figure}[htbp]
    \centering
    \begin{overpic}[width=0.9\textwidth]{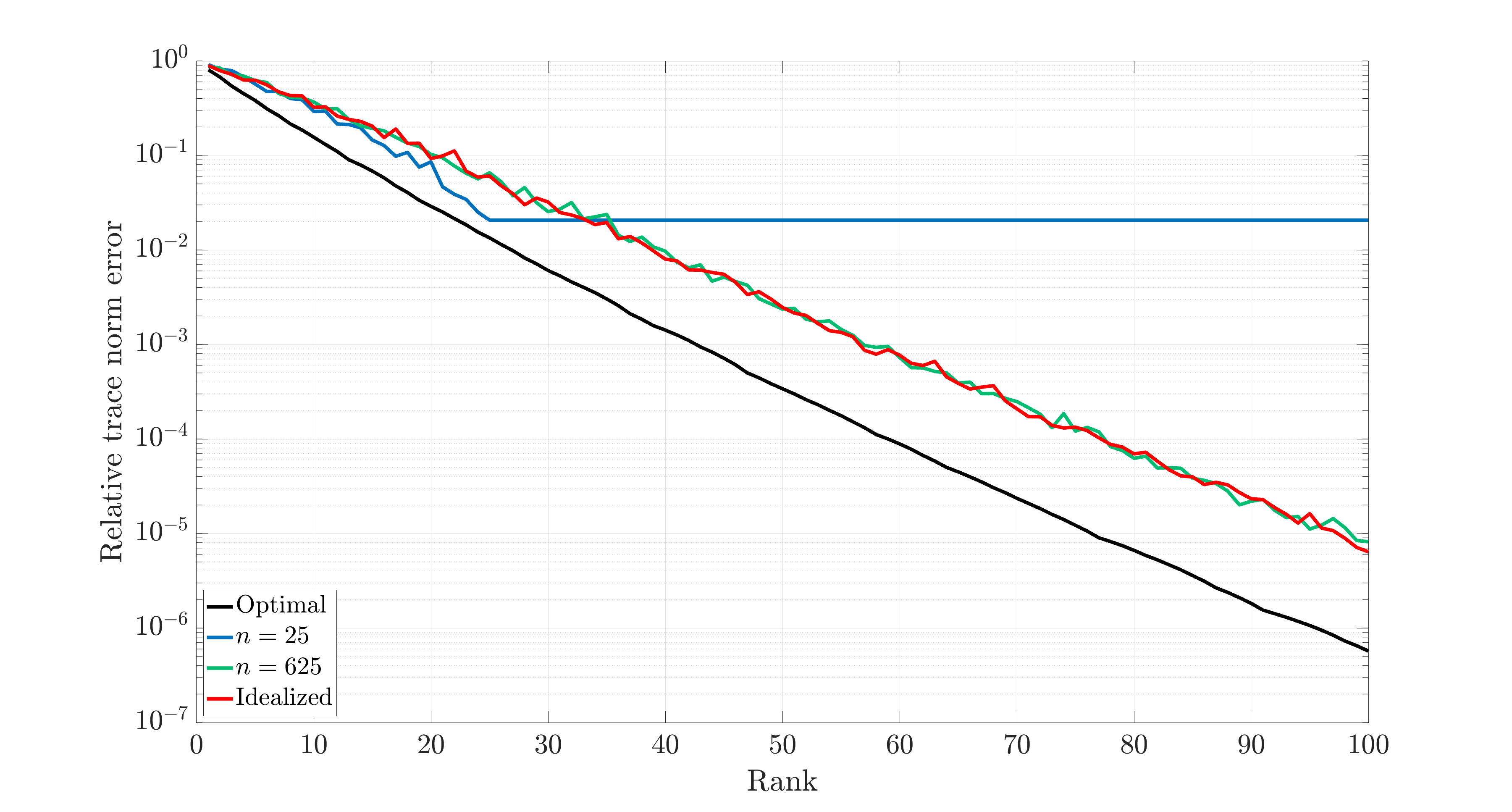}
    \end{overpic}
    \caption{Relative trace norm errors for approximations of the integral operator defined by \eqref{eq:sekernel}. \emph{Optimal} denotes the optimal low-rank approximation error given by the truncated SVD. $n$ refers to the truncation parameters from \Cref{section:finite_nystrom}. \emph{Idealized} refers to the idealized algorithm outlined in \Cref{alg:nystrom}. }
    \label{fig:gp}
\end{figure}

Interestingly, when $n = 25$ the error associated with the discrete randomized SVD is smaller than that for both $n = 625$ and the idealized Nyström approximation. This phenomenon has been noted previously and an explanation is given in \cite[Sec.~4]{perssonboullekressner2025}. When $n$ is small, the random fields have a small variance in the direction of $U_k^{\bot}$, which results in near-optimal approximations for small $k$. However, as $k$ increases, these random fields become too spatially biased, causing the error to stagnate. Moreover, when $n$ is chosen sufficiently large, the errors from the idealized and discrete Nyström approximations are nearly indistinguishable, as supported by \Cref{thm:convergenceexpect} and \Cref{theorem:nystromconvergence}. 

\rev{As a third experiment, we demonstrate how a poorly chosen covariance kernel can lead a stagnation of the approximation error in the setting of Boullé and Townsend \cite{boulle2022learning}, as discussed at the end of \Cref{sec:analysis_infrsvd}. We consider the integral operator defined by the kernel
\begin{equation}\label{eq:pretty}
    \kappa(x,y) = \frac{1}{1+100(x^2-y^2)^2}, \quad x,y \in [-1,1].
\end{equation}
The bases defining $Z_m$ and $W_n$ are chosen as orthonormalized Legendre polynomials. We perform the discrete randomized SVD, as outlined \Cref{section:fdapprox}, with $n = 258$ and $m$ is left to the internal Chebfun routine. We compare this approximation with the approach proposed by Boullé and Townsend \cite{boulle2022learning}, where the sketching vectors in $H_1$ are drawn from a Gaussian process with covariance operator $\mathcal{K}$; see also~\Cref{sec:analysis_infrsvd}. We first test a squared-exponential covariance operator with length parameter $0.02$. Specifically, we take
\begin{equation*}
    \mathcal{K}_{\text{SE}}: L^2([-1,1]) \to L^2([-1,1]), \quad [\mathcal{K}_{\text{SE}}f](x) = \frac{1}{0.02\sqrt{2\pi}}\int_{-1}^1 \exp\left(-\frac{(x-y)^2}{2\cdot0.02^2}\right)f(y)\,\mathrm{d}y.
\end{equation*}
The numerical rank of this operator, as determined by Chebfun, is $258$, and its eigenvalues decay superexponentially. To separate the effect of eigenvector alignment from the effect of eigenvector decay, we also test a synthetic covariance operator with eigendecomposition
\begin{equation*}
    \mathcal{K}_{\text{synth}} = W_n \Lambda W_n^*, \quad \Lambda_{i,i} = \exp\left(-0.25(i-1)\right), \quad i = 1,\ldots,258. 
\end{equation*}
The eigenvectors of $\mathcal{K}_{\text{synth}}$ coincides with the columns of the quasi-matrix $W_n$, but the eigenvalues decay rapidly. As shown in \Cref{fig:decay}, using $\mathcal{K}_{\text{SE}}$ and $\mathcal{K}_{\text{synth}}$ as covariance operators for the random fields leads to a clear stagnation of the approximation error. In other words, excessive smoothness in the sketch matrices can result in poor approximations. 

\begin{figure}
    \centering
    \begin{overpic}[width=0.9\textwidth]{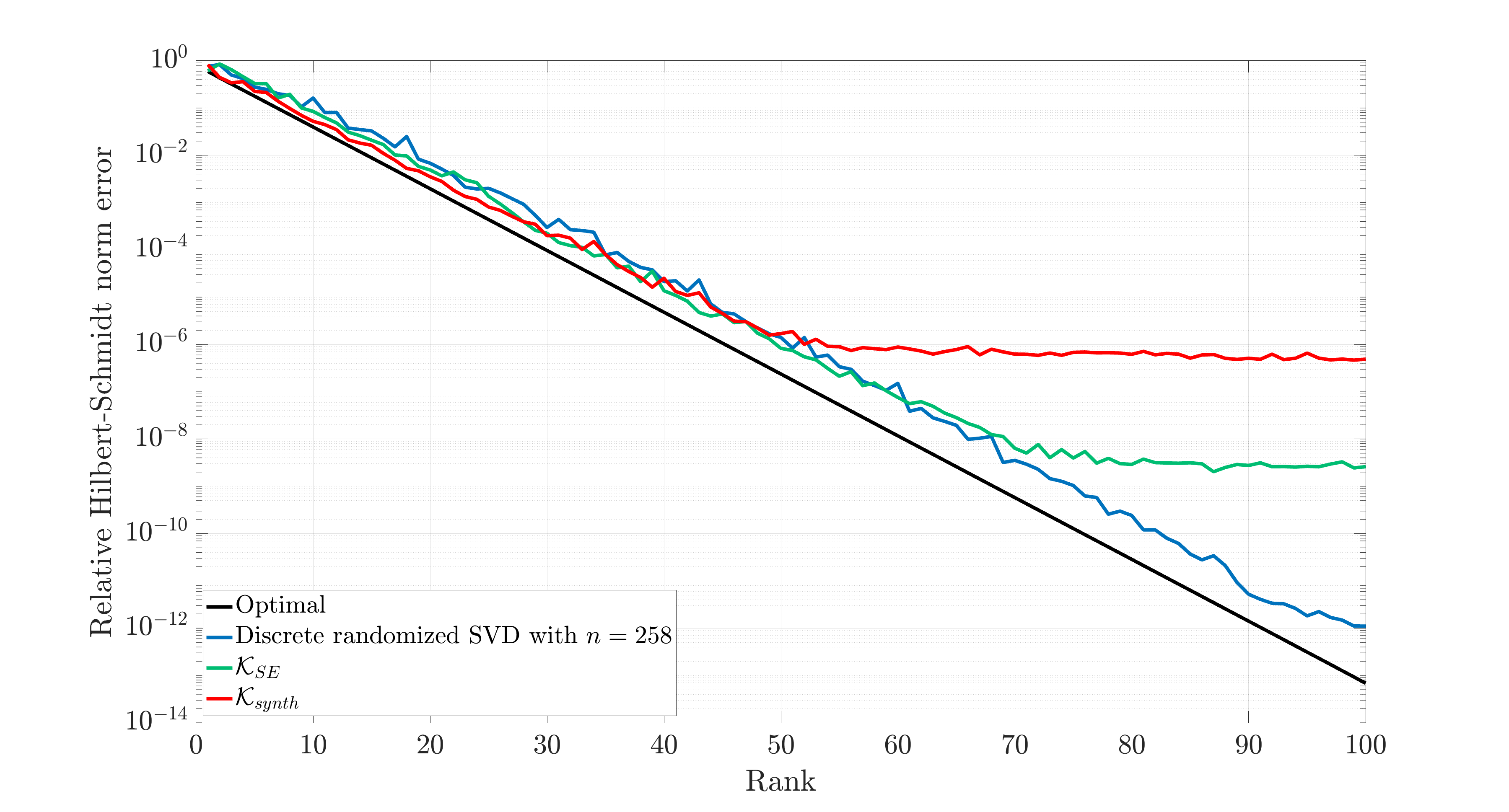}
    \end{overpic}
    \caption{\rev{Relative Hilbert-Schmidt norm errors for approximations of the integral operator defined by \eqref{eq:pretty}. \emph{Optimal} denotes the optimal low-rank approximation error given by the truncated SVD. \emph{Discrete randomized SVD with $n = 258$} corresponds to the discrete randomized SVD discussed in \Cref{section:fdapprox}. $\mathcal{K}_{\text{SE}}$ refers to the randomized SVD with random sketches drawn from $\mathcal{N}_{H_1}(0,\mathcal{K}_{\text{SE}})$. $\mathcal{K}_{\text{synth}}$ refers to the randomized SVD with random sketches drawn from $\mathcal{N}_{H_1}(0,\mathcal{K}_{\text{synth}})$.}
    }
    \label{fig:decay}
\end{figure}

}

\section{Conclusions}

In this work we obtained error bounds for an infinite-dimensional version of the randomized SVD for Hilbert--Schmidt operators. In the idealized version of this algorithm vectors are sampled in the image of the operator using an underlying Gaussian measure induced by the operator itself. With this viewpoint it is possible to obtain error bounds completely analogous to the finite-dimensional case. When applying the randomized SVD to discretizations of the operator, we showed how to decouple the discretization error from the randomized SVD error, which allowed to prove convergence of the randomized SVDs in a suitable sense based on Wasserstein distances. The results are further generalized to the Nystr\"om approximation of trace-class operators. For the practical realization we proposed a heuristic truncation scheme for resolving the range and co-range of a given operator in a suitable finite basis expansion, which showed convincing results in numerical experiments for simple integral kernels. The current analysis is limited to discretizations of the operator based on orthonormal basis expansions, which is often not a realistic scenario for numerical approximation. Future work could therefore concentrate on the interplay between discretization error and the randomized SVD error for more general and operator-adapted bases, and investigate the implications in the context of PDE learning.

\paragraph*{Acknowledgements}

The work of A.U.~was supported by the Deutsche Forschungs\-gemeinschaft (DFG, German Research Foundation) – Projektnummer 506561557.

{\small
\bibliographystyle{siam}
\bibliography{bibliography}
}

\appendix 

\section{Perturbation of orthogonal projectors}

The proof of~\Cref{thm:convergenceexpect} relies on the following perturbation result for orthogonal projectors.
\begin{lemma} \label{lemma:projpert}
Given a Hilbert space $H$ and quasi-matrices $Y, \widehat Y: \R^{r} \to H$, let
$P_Y, P_{\widehat Y} : H \to H$ denote the orthogonal projectors onto
the ranges of $Y$ and $\widehat Y$, respectively. Assuming that both
ranges have dimension $r$, it holds that
\[
 \| P_Y - P_{\widehat Y}  \|_{\op} \le \min\big\{ \|Y^\dagger\|_{\op}, \|\widehat Y^\dagger\|_{\op}  \big\}\cdot 
 \|Y - \widehat Y\|_{\op},
\]
where $\dagger$ is used to denote the pseudo-inverse $Y^\dagger = (Y^* Y)^{-1} Y^*$.
\end{lemma}
\begin{proof}
Note that the statement remains the same when the Hilbert space $H$ is replaced by a $2r$-dimensional subspace that contains the ranges of $Y, \widehat Y$. Thus, the result follows from the corresponding finite-dimensional result~\cite{Sun1984}; see also~\cite[Theorem 2.1]{ilsenogap}.
\end{proof}

\end{document}